\documentclass[11pt,a4paper]{article}
\usepackage{fancyhdr}
\usepackage[british]{babel}
\usepackage{algorithm}
\usepackage{algorithmic}
\usepackage{amssymb, amsmath, amsthm, amsfonts, enumerate,bm}
\usepackage{amsmath,setspace,scalefnt}
\usepackage[usenames,dvipsnames,svgnames,table]{xcolor}
\usepackage{graphicx,tikz,caption,subcaption}
\usetikzlibrary{patterns}
\usetikzlibrary{shapes}
\usepackage{fullpage}
\usepackage{hyperref}
\usepackage{enumitem,verbatim}
\usepackage{multicol}
\usepackage{float}  
\usepackage{cleveref}

\usepackage[margin=2cm]{geometry}

\definecolor{gray}{rgb}{0.25, 0.25, 0.25}

\usepackage{booktabs}  

\counterwithin{figure}{section}

\newtheorem{theorem}{Theorem}[section]
\newtheorem{lemma}[theorem]{Lemma}
\newtheorem{claim}[theorem]{Claim}
\newtheorem{corollary}[theorem]{Corollary}

\newtheorem{problem}[theorem]{Problem}

\newenvironment{definition*}
  {
   \innerdefinition}
  {\endinnerdefinition}

\theoremstyle{definition}
\newtheorem{defn}[theorem]{Definition}
\newtheorem*{defn*}{Definition}
\newtheorem*{observation*}{Observation}

\theoremstyle{remark}

\usepackage{todonotes}

\title{Strong spectral stabilities for $C_{2k+1}$-free graphs}

\author{
Lantao Zou\thanks{School of Mathematics, Hunan University, Changsha, Hunan, 410082, China.
Email: \url{zoulantao123@163.com}. }
\and
Yongtao Li\thanks{Corresponding author, Yau Mathematical Sciences Center, Tsinghua University, Beijing, China. Email: \url{yongtao_li@mail.tsinghua.edu.cn}. Supported by the Shuimu Scholar Program at Tsinghua University.}
\and
Yuejian Peng\thanks{School of Mathematics, Hunan University, Changsha, Hunan, 410082, China.
Email: \url{ypeng1@hnu.edu.cn}.
Supported by the NSF of Hunan Province (No. 2025JJ30003) and the NSFC grant (No. 11931002).}
}

\date{\today}

\begin{document}
\maketitle

\begin{abstract}
    The Tur\'{a}n-type problem for odd cycles is a classical problem in extremal graph theory.
    F\"{u}redi and Gunderson [Combin. Probab. Comput. 24 (4) (2015)] proved that if $k\ge 2$ and $n\ge 4k-2$, then every $n$-vertex $C_{2k+1}$-free graph contains at most $\lfloor n^2/4 \rfloor$ edges. Recently, a stability result due to  Ren, Wang, Wang and Yang [SIAM J. Discrete Math. 38 (2024)] shows that if $3\le r \le 2k$ and $n\ge 318 (r-2)^2k$, and $G$ is a $C_{2k+1}$-free graph on $n$ vertices with $e(G)\ge \lfloor {(n-r+1)^2}/{4}\rfloor +{r \choose 2}$, then $G$ can be made bipartite by deleting at most $r-2$ vertices. Using a different method, we give a linear bound on $n$ in terms of $k$ and show a stronger structural result, which roughly says that $G$ can be obtained from a large bipartite graph by suspending some small graphs that the total number of vertices is at most $r-2$.  This improves a result of Yan and Peng (2024) by weakening the requirement on $n$ and $k$.
    As a direct corollary, we obtain a tight upper bound  on the size of an $n$-vertex $C_{2k+1}$-free graph with chromatic number $\chi (G)\ge r$ for every $r\le 2k$.
   This corollary can be viewed as an analogue of a result of Erd\H{o}s (1955) and another result of Ren, Wang, Wang and Yang (2024) for triangle-free graphs with chromatic number at least $3$ and $4$, respectively.

    The second part of this paper concerns the spectral extremal problem for  $C_{2k+1}$-free graphs.
     We denote by $\lambda (G)$ the spectral radius of the adjacency matrix of a graph $G$.
    Let $T_{n-r+1,2}\circ K_r$ be the graph obtained by identifying a vertex of the complete graph $K_r$ and a vertex of the smaller partite set of the bipartite Tur\'{a}n graph $T_{n-r+1 ,2}$. Using the spectral techniques,
    we prove that if $3\le r\le 2k$ and $n\ge 712k$, and $G$ is an $n$-vertex $C_{2k+1}$-free graph with chromatic number $\chi (G) \ge r$, then $\lambda (G)\le \lambda (T_{n-r+1,2}\circ K_r)$, where the equality holds if and only if $G=T_{n-r+1,2}\circ K_r$.
    Our result not only extends a result of Guo, Lin and Zhao [Linear Algebra Appl. 627 (2021)] as well as a result of Zhang and Zhao [Discrete Math. 346 (2023)], but also provides the first solution to the spectral extremal problem for $F$-free graphs with high chromatic number. This is a significant advancement beyond previous work limited to non-bipartite graphs, as it gives a spectral condition to guarantee that a graph is structurally $(r-1)$-partite.

   The key ingredient in our approach is to show the existence of  a large  $2$-connected bipartite induced subgraph  with `nice' properties which allows us to have a clear picture on structures outside the block containing this $2$-connected bipartite induced subgraph. Our approach somehow provides a unified proof for both edge stability and spectral stability of $C_{2k+1}$-free graphs.

\end{abstract}



\section{Introduction}

 We follow the standard notations. For a simple undirected graph $G$, let $V(G)$, $E(G)$, $e(G)$, $\delta(G)$ and $\chi(G)$  denote the  vertex set of $G$,  the  edge set of  $G$, the size of $E(G)$, the minimum degree of $G$ and the chromatic number of $G$ respectively. A graph $G$ is called {\it $F$-free}  if it does not contain a subgraph isomorphic to $F$.
  For example, every bipartite graph is triangle-free.
 The {\em Tur\'{a}n number} of $F$, denoted by $\mathrm{ex}(n, F)$, is defined as the maximum number of edges
  in an $n$-vertex $F$-free graph.
  An $F$-free graph with the maximum number of edges is called an {\em extremal graph} for $F$.
 Let $T_{n,r}$ denote the complete $r$-partite graph on $n$ vertices
 whose part sizes are as equal as possible.
 The famous Tur\'{a}n theorem \cite{Turan41} states that
if $G$ is an $n$-vertex $K_{r+1}$-free graph,
then $e(G)\le e(T_{n,r})$,
and the equality holds if and only if $G$ is an $r$-partite Tur\'{a}n graph $T_{n,r}$.
The Tur\'{a}n theorem implies that
for every $K_{r+1}$-free graph $G$,
\begin{equation} \label{eq-Turan}
e(G)\le \left(1-\frac{1}{r} \right) \frac{n^2}{2}.
\end{equation}
Many alternative proofs of the Tur\'{a}n theorem can be found in the literature;
see \cite[pp. 269--273]{AZ2014} and \cite[pp. 294--301]{Bollobas78}.
There are various extensions of the Tur\'{a}n theorem; see, e.g., \cite{BT1981,Bon1983}.
The {\it chromatic number} of a graph $F$, denoted by $\chi (F)$, is the minimum number of colors required to color the vertices of $F$ such that no two adjacent vertices share the same color.
The  celebrated extension of
 Erd\H{o}s, Stone and Simonovits \cite{ES46,ES66} states that
 if $F$ is a graph with $\chi (F)=r+1$, then
 \[ \mathrm{ex}(n,F) =
 \left( 1- \frac{1}{r} + o(1) \right)\frac{n^2}{2}. \]
 This provides an asymptotic estimate for the Tur\'{a}n number of non-bipartite graphs.
 We refer the readers to the comprehensive surveys \cite{FS13,Sim2013} for related results on Tur\'{a}n type problems.

In 1966,
Erd\H{o}s \cite{Erd1966Sta1,Erd1966Sta2} and Simonovits \cite{Sim1966} proved
a  structural theorem, which shows that
the Tur\'{a}n problem exhibits a certain stability phenomenon.
The Erd\H{o}s--Simonovits stability says that
for any $\varepsilon >0$ and any graph $F$ with $\chi (F)=r+1$, there exist $n_0$ and $\delta >0$ such that
if $G$ is an $F$-free graph on $n\ge n_0$ vertices with
\[ e(G)\ge \left(1- \frac{1}{r} - \delta \right)\frac{n^2}{2}, \]
then $G$ can be made $r$-partite by removing at most $\varepsilon n^2$ edges.
This stability theorem has become a powerful tool for solving the Tur\'{a}n numbers of some classical graphs; see \cite[Section 5]{Kee2011}.
In 2015,  F\"{u}redi \cite{Fur2015} proved a more precise  stability result for cliques, which states that
if $G$ is an $n$-vertex $K_{r+1}$-free graph with
$e(G)\ge e(T_{n,r}) - t$ edges, then $G$ can be made $r$-partite by removing at most $t$ edges.
This result also gives a {\it precise linear dependency} of $\delta $ in terms of $\varepsilon$, and drops the condition that the order $n$ is sufficiently large.
 Further stability results for cliques have been further well-studied in the past few years;  see \cite{RS2018,BCLLP2021,KRS2021} and the references therein.

\subsection{Strong edge stability  for $C_{2k+1}$-free graphs}

In the study of graph theory, cycles are an extremely important class of graphs. The study of cycles plays a fundamental role in both structural graph theory and extremal graph theory; see the survey \cite{Ver2016} and the recent progress \cite{LM2023,GCS2024,CJMM2025}.
In this paper, we focus mainly on the study of extremal graph problems and stabilities for odd cycles.
The Tur\'{a}n number $\mathrm{ex}(n,C_{2k+1})$ could be determined by the celebrated stability method \cite{Sim1966} and \cite[Sec. 5]{Kee2011}.
However, the stability method works only for sufficiently large $n$. In fact, the value of $\mathrm{ex}(n,C_{2k+1})$ can be extracted from the works of Bondy \cite{Bon1971jctb,Bon1971dm}, Woodall \cite{Woo1972} and Bollob\'{a}s \cite[pp. 147--156]{Bollobas78}.
Later, F\"{u}redi and Gunderson \cite{FG2015cpc} proved that
    for every $k\ge 2$ and $n\ge 4k-2$,
    \[ \mathrm{ex}(n,C_{2k+1})= \left\lfloor \frac{n^2}{4} \right\rfloor . \]
    Moreover, for $n\ge 4k$, the unique extremal graph is the bipartite Tur\'{a}n graph $T_{n,2}$.

A stability result for odd cycles was proved by Li, Feng and Peng \cite[Theorem 4.8]{LFP2024-triangular}. It states that for every $k\ge 1$ and $0<\varepsilon < \frac{1}{2}$,
if $G$ is a $C_{2k+1}$-free graph on ${n\ge {2k}/{\varepsilon}}$ vertices
with $e(G)\ge (1 - 2{\varepsilon} ) \frac{n^2}{4}$,
then $G$ can be made bipartite by deleting at most $\varepsilon n^2$ edges.
In 2024, Ren, Wang, Wang and Yang \cite{RWWY2024-SIAM} established a  stability {when $e(G)$ is quite close to $e(T_{n,2})$}, i.e., within $O(n)$ of $e(T_{n,2})$, by showing that if $k\ge 2$ and $G$ is a $C_{2k+1}$-free graph with $e(G)\ge
\lfloor \frac{(n-r+1)^2}{4} \rfloor + {r \choose 2}$, where $3\le r\le 2k$, then $G$ can be made bipartite by removing at most ${r^2}/{4}$ edges.
Moreover, they proved a stability result by removing a few vertices to make the graph $G$ bipartite.

\begin{theorem}[Ren, Wang, Wang and Yang \cite{RWWY2024-SIAM}] \label{thm-RWWY}
Let $n,k,r$ be integers with $k\ge 2, 3\le r\le 2k$ and $n\ge 318 (r-2)^2k$. If $G$ is a $C_{2k+1}$-free graph on $n$ vertices with
\[ e(G)\ge \left\lfloor \frac{(n-r+1)^2}{4} \right\rfloor + {r \choose 2}, \]
then $G$ can be made bipartite by {deleting at most $r-2$ vertices}. The extremal graph is obtained from $T_{n-r+1,2}$ and $K_r$ by sharing exactly one vertex.
\end{theorem}

Theorem \ref{thm-RWWY} was recently studied by Wang and Wang \cite{WW2025} under the condition $3\le r \le 2k$ and $n\ge 2(r+1)(r+2)(r+2k)$.  {Yan and Peng \cite{Yan-Peng2024} introduced a new concept `strong-$2k$-core' and provided a concise proof to  obtain a stronger structural stability for $r\le 2k-4$ and $n\ge 20(r+2)^2k$}.  We say that a graph $G$ is suspended on a graph $B$ if $G$ intersects with $B$ in exactly one vertex, see Figure \ref{Fig-Extremal-sus}.

 \begin{figure}[H]
\centering
\includegraphics[scale=0.8]{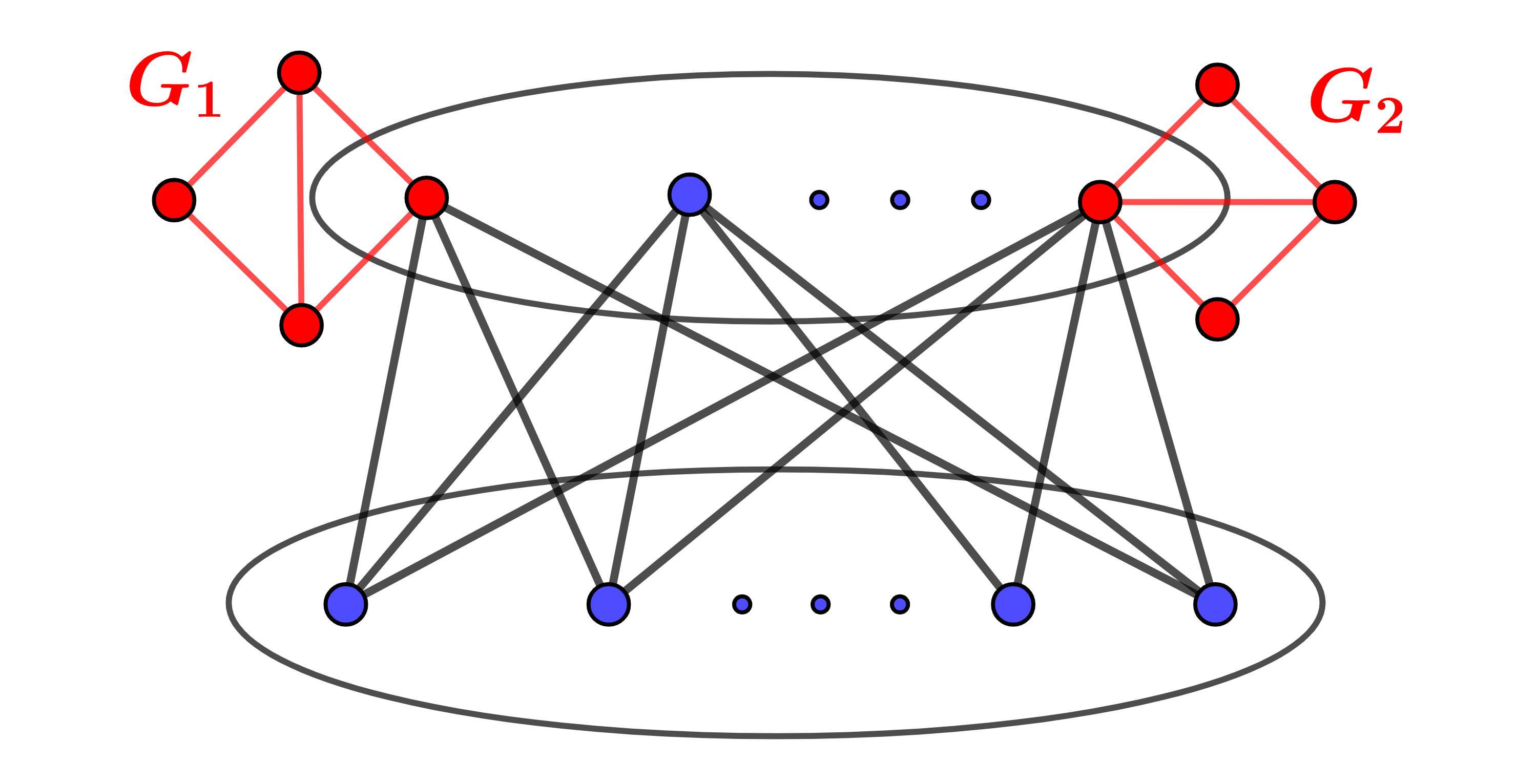}
\caption{A bipartite graph $B$ suspending $G_1$ and $G_2$.}
\label{Fig-Extremal-sus}
\end{figure}

\begin{defn*}
    Let \( \mathcal{G}_{n,r} \) be the family of \( n \)-vertex graphs that are obtained from a bipartite graph \( B \) by suspending graphs \( G_1, G_2,\ldots, G_s \) for some $s\in \mathbb{N}$, where \( \sum_{i=1}^s |V(G_i) \setminus V(B)|\le r-2 \).
\end{defn*}

\begin{theorem}[Yan and Peng \cite{Yan-Peng2024}]\label{thm-YP}
Let $r\ge 3$,  $2k\ge r+4$ and $n\ge 20(r+2)^2k$ be integers. Suppose that $G$ is an $n$-vertex $C_{2k+1}$-free graph with $e(G)\ge \big\lfloor{\frac{(n-r+1)^2}{4}}\big\rfloor+{r\choose 2} $.    Then $G\in \mathcal{G}_{n,r}$, unless $G$ is obtained from the bipartite Tur\'{a}n graph $T_{n-r+1,2}$ by suspending a clique $K_r$.
\end{theorem}

Using a new method different from methods in \cite{RWWY2024-SIAM, Yan-Peng2024}, we weaken  the condition $2k\ge r+4$ to $2k\ge r$, and weaken $n\ge 20(r+2)^2k$ to $n\ge 100k$ in Theorem \ref{thm-YP}.  Furthermore, our method works not only for edge stability of $C_{2k+1}$-free graphs, it also allows us to obtain strong spectral stability of $C_{2k+1}$-free graphs (Theorem \ref{thm-chromatic-spectral-version}). Let us state the first main result in this paper precisely.

\begin{theorem}\label{thm-chromatic-edges-version}
Let  $k\ge 2$,  $3 \le r\le 2k$ and $n\ge 100k$. Suppose that $G$ is an $n$-vertex $C_{2k+1}$-free graph with $e(G)\ge \big\lfloor{\frac{(n-r+1)^2}{4}}\big\rfloor+{r\choose 2} $.   Then $G\in \mathcal{G}_{n,r}$, unless $G$ is obtained from $T_{n-r+1,2}$ by suspending $K_r$.
\end{theorem}

Theorems \ref{thm-YP} and \ref{thm-chromatic-edges-version} provide a more detailed structure, instead of making it bipartite by deleting at most $r-2$ vertices.
In our result, we weaken the constraint of Theorems \ref{thm-RWWY} and \ref{thm-YP} to the linear dependence $n\ge 100k$. This improvement is interesting in view of the following corollary.

\begin{corollary}
Suppose that $r\ge 3$ and $G$ is an $n$-vertex graph with $e(G)\ge \big\lfloor{\frac{(n-r+1)^2}{4}}\big\rfloor+{r\choose 2}$.   If  $G\notin \mathcal{G}_{n,r}$, and $G$ is not obtained from $T_{n-r+1,2}$ by suspending $K_r$, then $C_{2k+1}\subset G$ for $\frac{r}{2} \le k\le {n \over 100}$.
\end{corollary}

\subsection{$C_{2k+1}$-free graphs with high chromatic number}

 The Mantel theorem \cite{Bollobas78} states that every $n$-vertex triangle-free graph contains at most $\lfloor n^2/4\rfloor$ edges.
An old result of Erd\H{o}s \cite[p. 306, Exercise 12.2.7]{BM2008} states that
if $G$ is a non-bipartite triangle-free graph  on $n$ vertices, then
\[ e(G)\le \left\lfloor \frac{(n-1)^2}{4} \right\rfloor+1 .\]
There are many different extremal graphs that achieve this bound as seen by blowing up the vertices of a $5$-cycle to independent sets of size $1,1,a,b,c$, respectively, where $a+c=\lfloor \frac{n}{2} \rfloor$ and $b=\lceil\frac{n}{2} \rceil -2$.
Note that $G$ is non-bipartite if and only if the chromatic number $\chi (G)\ge 3$. In another paper, Ren, Wang, Wang and Yang \cite{RWWY2024} proved the following result.

\begin{theorem}[Ren, Wang, Wang and Yang \cite{RWWY2024}]
   If $G$ is a triangle-free graph on $n\ge 150$ vertices
   with chromatic number $\chi (G)\ge 4$, then
   \[ e(G)\le \left\lfloor \frac{(n-3)^2}{4} \right\rfloor + 5,\]
with equality if and only if $G$ is a specific blow-up of the Gr\"{o}tzsch graph.
\end{theorem}

We point out that it is difficult to determine the maximum size of a triangle-free graph $G$ under the constraint $\chi (G)\ge r$ for a large integer $r$. One of the primary challenges in addressing this problem lies in finding the smallest order of a triangle-free graph with chromatic number at least $r$.
Indeed, let $n(r)$ be the smallest order of a triangle-free graph $G$ such that $\chi (G)\ge r$. It is known that $n(3)=5$ as witnessed by the $5$-cycle, $n(4)=11$ by the Gr\"{o}tzsch graph, and $n(5)=22$ by the `complicated' graph constructed in \cite{JR1995}. It is quite difficult to determine the exact value of $n(r)$ for an integer $r$.


 As the second corollary of
 Theorem \ref{thm-chromatic-edges-version},
 we give a tight upper bound on the number of edges of an $n$-veretx $C_{2k+1}$-free graph $G$ with chromatic number $\chi (G)\ge r$.

\begin{corollary} \label{thm-edge}
    Let $k\ge 2, 3\le r\le 2k$ and $n\ge 100k$.  If $G$ is an $n$-vertex $C_{2k+1}$-free graph with $\chi (G)\ge r$, then
    \[ e(G)\le \left\lfloor{\frac{(n-r+1)^2}{4}}\right\rfloor +{r\choose 2}, \]
    with equality if and only if $G$ is obtained from $T_{n-r+1,2}$ and $K_r$ by sharing exactly one vertex.
\end{corollary}

In what follows, we explain that Corollary \ref{thm-edge} is a direct consequence of Theorem \ref{thm-chromatic-edges-version}.
 
\begin{proof}
Applying Theorem \ref{thm-chromatic-edges-version},
we have that  if  $n\ge 100k$ and $G$ is an $n$-vertex $C_{2k+1}$-free graph with $e(G)\ge \big\lfloor{\frac{(n-r+1)^2}{4}}\big\rfloor+{r\choose 2} $, then either $G$ can be obtained from a bipartite graph $B$ by suspending graphs $G_1, G_2,\ldots, G_s$ such that $\sum_{i=1}^s |V(G_i)\setminus V(B)|\le r-2$, or $G$ is obtained by suspending $K_r$ to $T_{n-r+1,2}$.
In the former case,  we can color the vertices of the bipartite graph $B$ by two colors. Since each $G_i$  is suspended on the bipartite graph $B$,  we can properly color the vertices of $V(G_i)\setminus V(B)$ by at most $|V(G_i)\setminus V(B)|-1$  new colors. Therefore, we can properly color the vertices of $G$ by at most $2+\sum_{i=1}^s (|V(G_i)\setminus V(B)|-1)\le r-1$ colors.
In conclusion, the above argument implies that if
$G$ is an $n$-vertex $C_{2k+1}$-free graph with $e(G)\ge \lfloor \frac{(n-r+1)^2}{4} \rfloor + {r \choose 2}$ and $\chi (G)\ge r$, then $G$ is obtained from $T_{n-r+1,2}$ by suspending a copy of $K_r$.
This completes the proof of Corollary \ref{thm-edge}.
\end{proof}

\subsection{ Strong spectral  stability for  $C_{2k+1}$-free graphs}

In recent years, there has been significant and popular growth in the field of spectral extremal graph theory.
Spectral extremal graph theory is an advanced and specialized area of mathematics that combines graph theory with algebraic methods, particularly through the use of rank and eigenvalues of the associated matrices to study the combinatorial structures of graphs. This field has gained popularity in recent years because of its potential applications in various domains such as computer science, engineering, and more broadly, in the study of complex networks and systems.

Let $G$ be a simple graph on the vertex set $\{v_1,v_2,\ldots ,v_n\}$. The adjacency matrix of
$G$ is defined as $A(G)=[a_{i,j}]_{i,j=1}^n$,
where $a_{i,j}=1$ if $v_i$ and $v_j$ are adjacent, and $a_{i,j}=0$ otherwise.
Let $\lambda (G)$ be the spectral radius
of $G$, which is defined as the maximum modulus of eigenvalues of $A(G)$.
 Since $A(G)$ is a non-negative matrix,
the Perron--Frobenius theorem implies that $\lambda (G)$ is actually the largest eigenvalue of $A(G)$.
The study in this article mainly concentrates on the extremal problems of the spectral radius of the adjacency matrix of a graph.

In 1986,  Wilf  \cite{Wil1986}  provided the first result regarding the spectral version of Tur\'{a}n's theorem and proved that for every $n$-vertex $K_{r+1}$-free graph $G$, we have
\begin{equation} \label{eq-Wilf}
\lambda (G)\le \left(1-\frac{1}{r} \right)n.
\end{equation}
In 2002, Nikiforov \cite{Niki2002cpc}
proved that if $G$ is a $K_{r+1}$-free graph with $m$ edges, then
\begin{equation} \label{eq-Niki-2m}
\lambda (G)^2 \le \left( 1-\frac{1}{r} \right)2m.
\end{equation}
 The case of equality was later characterized in \cite{Niki2006-walks, Niki2009-more}.
Either  the bound in (\ref{eq-Wilf}) or (\ref{eq-Niki-2m}) can imply the Tur\'{a}n bound in (\ref{eq-Turan}).
Furthermore, Nikiforov \cite{Niki2007laa2} and Guiduli \cite{Gui1996} independently sharpened Wilf's bound (\ref{eq-Wilf}) and showed a
precise version of the spectral Tur\'{a}n theorem, which states that if $G$ is an $n$-vertex $K_{r+1}$-free graph, then $\lambda (G)\le \lambda ({T_{n,r}})$, and the equality holds if and only if $G=T_{n,r}$.

In 2008,
Nikiforov \cite{Niki2008} studied the spectral extremal problem for odd cycles and proved that if $k\ge 2$ and $n\ge 320(2k+1)$ is sufficiently large,
and $G$ is a $C_{2k+1}$-free graph of order $n$, then
$\lambda (G)\le \lambda (T_{n,2})$.
To eliminate the condition that requires $n$ to be sufficiently large,  a recent work of Zhai and Lin \cite[Corollary 1.2]{ZL2022jgt} implies that
    if $n\ge 14k+7$ and $G$ is a $C_{2k+1}$-free graph on $n$ vertices, then $\lambda (G)\le \lambda (T_{n,2})$, with equality if and only if $G=T_{n,2}$.
 Note that the extremal triangle-free graphs are bipartite.
 Furthermore,
  Lin, Ning and Wu \cite{LNW2021} investigated the maximum spectral radius of triangle-free graphs by excluding all bipartite graphs.
Let $S(T_{n-1,2})$ be obtained from
$T_{n-1,2}$ by subdividing an edge.
It was shown in \cite{LNW2021} that if $G$ is triangle-free and non-bipartite, then
\[ \lambda (G)\le \lambda (S(T_{n-1,2})),\]
where the equality holds if and only if $G=  S(T_{n-1,2})$.
Inspired by this result, Li and Peng \cite{LP2022second} studied the non-$r$-partite $K_{r+1}$-free graphs.
These results also stimulated more studies of spectral extremal problems for non-bipartite graphs without certain subgraphs; see \cite{LG2021,LSY2022,LP-ejc2022,LFP2023-solution,ZFL2024} for odd cycles.

It is interesting to study the spectral problem for non-bipartite graphs without an odd cycle $C_{2k+1}$.
The case $k=2$ was initially proved by Guo, Lin and Zhao \cite{GLZ2021} for every $n\ge 21$; the general case for $k\ge 3$ was recently studied by Zhang and Zhao \cite{ZZ2023} for fixed $k$ and sufficiently large $n$. To state their result,
we denote by $T_{n-2,2}\circ K_3$ the graph obtained by identifying a vertex of $T_{n-2,2}$ that belongs to the partite set of size $\lfloor \frac{n-2}{2} \rfloor$ and a vertex of $K_3$.

\begin{theorem}[Guo, Lin and Zhao  \cite{GLZ2021}; Zhang and Zhao \cite{ZZ2023}]
\label{thm-non-bi-C2k+1}
     Let $k\geq 2$ and $n\ge 21$, or $k\ge 3$ and $n$ be sufficiently large.
If $G$ is a non-bipartite $C_{2k+1}$-free graph on $n$ vertices, then
     $$
     \lambda(G)\le \lambda(T_{n-2,2}\circ K_3),
     $$
     and the equality holds if and only if $G=  T_{n-2,2}\circ K_3$.
\end{theorem}

The second part of this paper studies {further} spectral extremal problems for $C_{2k+1}$-free graphs.
Using an almost unified argument, we shall establish the spectral versions of Theorem \ref{thm-chromatic-edges-version} and  Corollary \ref{thm-edge}.
To begin with,
 we introduce the definition of the spectral extremal graph.

 \begin{defn*}
     Let $T_{n-r+1,2}\circ K_r$ be the $n$-vertex graph obtained by identifying a vertex of $K_r$ and a vertex of the smaller partite set of $T_{n-r+1 ,2}$; see Figure \ref{Fig-Extremal-2}.
 \end{defn*}

 \begin{figure}[H]
\centering
\includegraphics[scale=0.8]{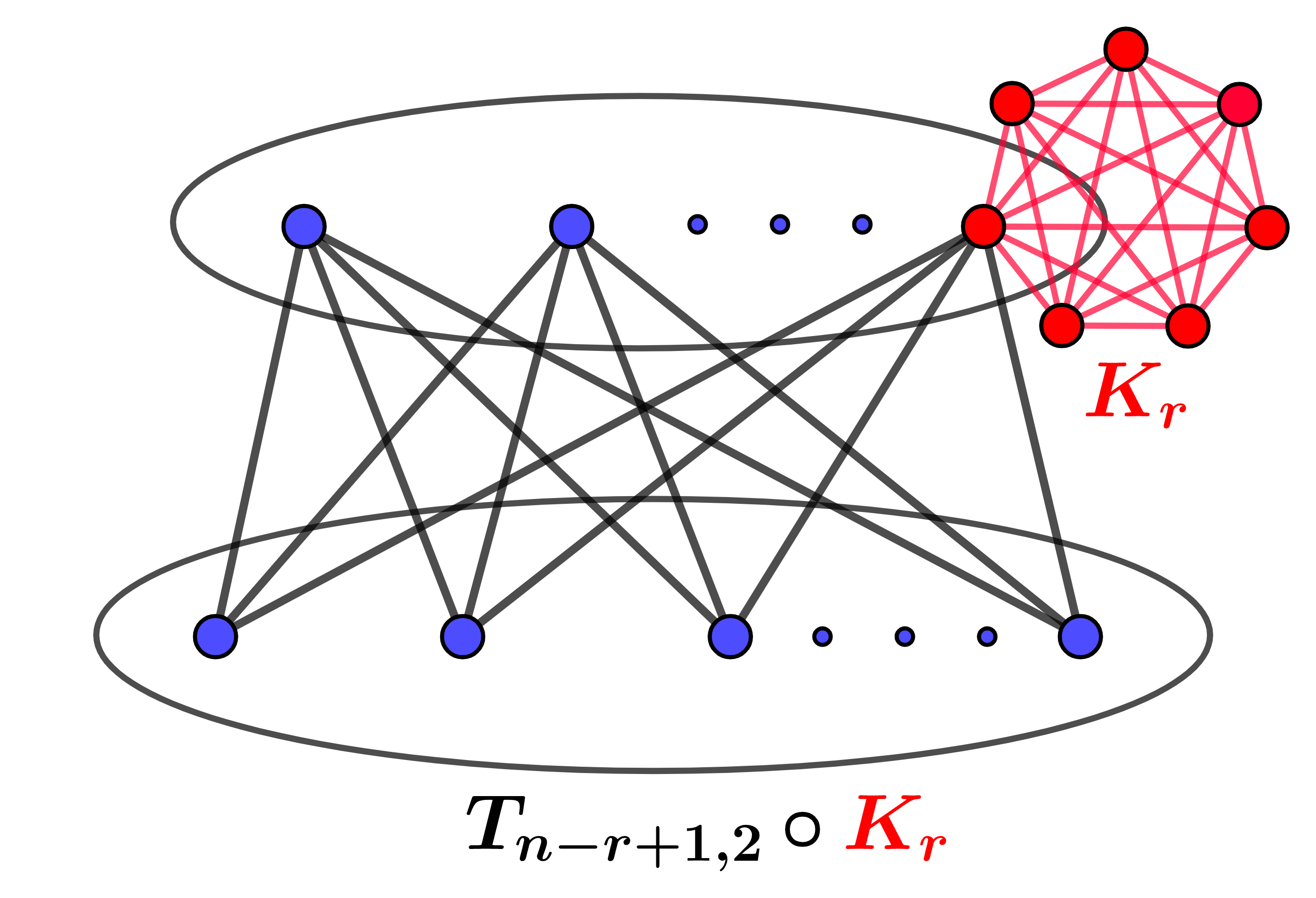}
\caption{The extremal graphs in Theorem \ref{thm-chromatic-spectral-version}.}
\label{Fig-Extremal-2}
\end{figure}

Our second main result shows that $C_{2k+1}$-free graphs have strong stability in terms of spectral radius.

\begin{theorem} \label{thm-chromatic-spectral-version}
Let $ k\geq 2, 3\le r\le 2k$ and $n\geq 712k$. Suppose that $G$ is an $n$-vertex $C_{2k+1}$-free graph with $\lambda(G)\geq \lambda(
T_{n-r+1,2} \circ K_r)$. Then $G\in \mathcal{G}_{n,r}$, unless $G=T_{n-r+1,2}\circ K_r$.
\end{theorem}

Theorem \ref{thm-chromatic-spectral-version} provides a spectral analogue of Theorem \ref{thm-chromatic-edges-version}.
As an application of Theorem \ref{thm-chromatic-spectral-version}, we can obtain the following extension of Theorem \ref{thm-non-bi-C2k+1}.

\begin{corollary} \label{thm-extension}
    Let $k\ge 2,3\le r\le 2k$ and $n\ge 712k$ be integers.  If $G$ is an $n$-vertex $C_{2k+1}$-free graph with chromatic number $\chi (G)\ge r$, then
    \[ \lambda (G)\le  \lambda (T_{n-r+1,2}\circ K_r), \]
    with equality if and only if $G=T_{n-r+1,2}\circ K_r$.
\end{corollary}

Corollary \ref{thm-extension} confirms a conjecture proposed by Yuejian Peng in Jiang's master thesis \cite{JP2025}.
Seeking linear bounds, or more precisely, exact thresholds, is regarded as an important and interesting research problem in spectral extremal graph theory; see \cite{LN2021outplanar,ZL2022jgt,LFP2024-triangular,LFP-count-bowtie}.
Our results in Theorem \ref{thm-chromatic-spectral-version} and  Corollary \ref{thm-extension} provide a linear bound on $n$, i.e., $n\ge 712k$, rather than requiring $n$ to be sufficiently large with respect to a fixed $k$.
 Our primary focus was on maintaining the generality and clarity of the argument rather than optimizing the numerical coefficient. We note that a more careful analysis and refinement of various components of our proof could lead to an improvement in this constant.

In recent years, numerous studies have explored the spectral extremal problems for non-bipartite graphs, as seen in \cite{LG2021,LSY2022,GLZ2021,
ZZ2023,LSW2023,LM2025,LLZ2025,FL2025} and references therein.
Undoubtedly, it seems difficult to study the spectral extremal problems for $F$-free graphs with high chromatic number.
We remark that Corollary \ref{thm-extension} could be seen as the first spectral extremal result that studies the maximum spectral radius of an $F$-free graph with high chromatic number.
Equivalently, under the condition that $G$ is $C_{2k+1}$-free and $\lambda (G)> \lambda (T_{n-r+1}\circ K_r)$, then $G$ must be $(r-1)$-partite. This could be considered to be a strong spectral stability result. We believe that it could provide a new understanding of the spectral extremal graph theory and opens up an interesting direction for our future research.

\paragraph{Our approach.}
Let us now briefly explain how our approach differs from the existing methods and the difficulties we encounter. 
Assume that $G$ is an $n$-vertex \( C_{2k+1} \)-free graph with $e(G)\ge \lfloor \frac{(n-r+1)^2}{4} \rfloor + {r \choose 2}$.
We need to show that $G$ is obtained from a bipartite graph by suspending some small graphs.
The key ingredient is to show the existence of a large `dense' $2$-connected bipartite induced subgraph \( G'\) by removing a few number of  vertices with small degree. We shall show that $G'$ has some nice properties; e.g., $G'$ has large order and large minimum degree. Moreover, any two vertices of $G'$ can be joined by a path with every certain length. For simplicity, we say that $G'$ is a \( k \)-dense bipartite subgraph (Definition \ref{def-dense} and Lemma \ref{lem-k-dense-bi-sub}).
 A block of $G$ is  a maximal $2$-connected subgraph of $G$.
The second part of our approach focuses on investigating the structure outside $G'$. The key ingredient in this part is to consider the block of $G$ that contains $G'$ as a subgraph and it allows us to obtain structural information on $G$.
Let $B$ be such a block.
We divide the argument into two cases.
(Case 1.)~Suppose that $B$ is bipartite, then \( G \) exhibits the desired structure, i.e., a bipartite graph $B$ by suspending some graphs, so we are done. (Case 2.)~Suppose that $B$ is not bipartite.  Then we can find a `bad' path $P_{uv}$ whose endpoints $u,v$ are in $G'$ and all other vertices are not in $G'$ (Lemma \ref{lem-bad-for-H}). Finally, if $P_{uv}$ has a small length, then $P_{uv}$ can be extended to an odd cycle $C_{2k+1}$.
Indeed, since $G'$ is a $k$-dense bipartite subgraph, we can find a path in $G'$ starting from $u$ to $v$, which together with $P_{uv}$ yields a copy of $C_{2k+1}$, a contradiction; if $P_{uv}$ has a large length, then by some structural analysis, we can prove that $e(G)< \lfloor \frac{(n-r+1)^2}{4} \rfloor + {r \choose 2}$, which contradicts with the assumption.
In other words, Case 2 does not occur. Therefore, the block of $G$ that contains $G'$ as a subgraph must be bipartite, and \( G \) is obtained by suspending some small graphs to the large bipartite block \( B \).

 Our method somehow provides a unified proof for Theorems \ref{thm-chromatic-edges-version} and  \ref{thm-extension}.  Based on two structural stability results (Lemmas \ref{lem-k-dense-bi-sub} and \ref{lem-bad-for-H}) established in proving Theorem \ref{thm-chromatic-edges-version}, a result of the Brualdi--Hoffman--Tur\'{a}n problem on cycles (Lemma \ref{lem-ZLS}), which allows us to transfer a bound on the spectral radius into a bound on the number of edges, and some other spectral techniques, e.g., the Sun--Das inequality in Lemma \ref{lem-Sun-Das-deletion} and the comparison principle in Lemma \ref{lem-WXH}, we  prove Theorem \ref{thm-extension}.

 \paragraph{More related background.}
In addition to the study on the size and the spectral radius of a graph, there is a rich history of investigating the Tur\'{a}n type problem in terms of the minimum degree.
An extremely difficult and widely investigated problem of Erd\H{o}s and Simonovits \cite{ES1973} asks for the maximum value of the minimum degree in an $n$-vertex $F$-free graph $G$ with chromatic number $\chi (G)\ge t$ for some integer $t$. This problem is also referred to as the chromatic profile.
In other words, an $F$-free graph that satisfies an appropriate condition on minimum degree has a bounded chromatic number. 
This type of result can be referred to as the degree-stability, which is stronger than the classical edge-stability.
We now list some related progress for convenience of the readers.

It is well-known that every non-bipartite triangle-free graph $G$ has the min-degree $\delta (G)\le \frac{2}{5}n$.
Moreover, Jin \cite{Jin1995} proved that if $G$ is a triangle-free graph with $\chi (G)\ge 4$, then $\delta (G)\le \frac{10}{29}n$.
A breakthrough of Brandt and Thomass\'{e} \cite{BT2005} states that if $G$ is triangle-free and $\chi (G)\ge 5$, then $\delta (G)\le \frac{n}{3}$.
This bound on minimum degree cannot be improved by restricting the chromatic number, since a construction of Hajnal shows that
there exists an $n$-vertex triangle-free graph $G$ with $\delta (G)=(\frac{1}{3}- o(1))n$, while $G$ has arbitrarily high chromatic number.
 {For cliques}, Andr\'{a}sfai, Erd\H{o}s and S\'{o}s \cite{AES1974} proved that if $G$ is a non-$r$-partite $K_{r+1}$-free graph, then $\delta (G)\le \frac{3r-4}{3r-1}n$. We refer to \cite{Bra2003} for a simple proof. Goddard and Lyle \cite{GL2011} extended the results of Jin, Brandt and Thomass\'{e} by showing that if $G$ is $K_{r+1}$-free, then $\chi (G)\ge r+2$ implies $\delta (G)\le \frac{19r-28}{19r-9}n$, and $\chi (G)\ge r+3$ implies $\delta (G)\le \frac{2r-3}{2r-1}n$.
This result was independently studied by Nikiforov \cite{Niki2010}.

It was also shown by Andr\'{a}sfai, Erd\H{o}s and S\'{o}s  \cite{AES1974}
 that if $n\ge 2k+3$ and $G$ is a non-bipartite $\{C_3,C_5,\ldots ,C_{2k+1}\}$-free graph of order $n$,
 then $\delta (G)\le \frac{2}{2k+3}n$.
Forbidding a single odd cycle, H\"{a}ggkvist \cite{Hagg1982} showed that for $k\in \{2,3,4\}$, every non-bipartite $C_{2k+1}$-free graph $G$ satisfies $\delta (G)\le \frac{2}{2k+3}n$.
In 2024, Yuan and Peng \cite{YP2024}
studied the case $k\ge 5$ and
proved that for $k\ge 5$ and $n\ge 21000k$, every non-bipartite $C_{2k+1}$-free graph has $\delta (G)\le \frac{n}{6}$. Furthermore, they \cite{YP2025} resolved the problem for forbidden families consisting of some odd cycles. Later, Yan, Peng and Yuan \cite{Yan-Peng-Yuan2024} showed that for $r\ge 4$, $k\ge 3r+1$ and $n\ge 108r^{r-1}k$, if $G$ is an $n$-vertex $C_{2k+1}$-free graph with $\chi (G)\ge r$, then $\delta (G)\le \frac{n}{2r}$.
We mention that all the above bounds on the minimum degree parameter are optimal.
For the sake of brevity, we do not describe the detailed extremal graphs.

 \paragraph{Organization.}
In Section \ref{sec-3}, we introduce some preliminary results {to be applied}. In Section \ref{sec-4}, we show some crucial structural results. For instance, we shall show that the desired extremal graph contains a large $k$-dense bipartite subgraph (Definition \ref{def-dense}). In Sections \ref{sec-5} and \ref{sec-6}, we give the proofs of Theorems \ref{thm-chromatic-edges-version} and \ref{thm-chromatic-spectral-version}, respectively.
In Appendix \ref{App-A}, we provide the proofs of Lemmas \ref{lem-spectral-order1} and \ref{lem-spectral-order2}.

 \paragraph{Notation.}
For a graph $G$ and a  vertex $v\in V(G)$, let $N(v)$ be the set of neighbors of $v$ in $G$. The {\it degree} $d(v)$ of $v$ is equal to $|N(v)|$.
 For $V_1, V_2\subseteq V( G)$, let $E( V_1, V_2)$ denote the set of edges of $G$ between $V_1$ and $V_2$, and let $e(V_1,V_2)=|E(V_1,V_2)|$. For every  $S\subseteq V(G)$, we denote $d_S(v)=|N_S(v)|=|N(v)\cap S|$.
 Let $G[S]$ be the graph induced by $S$ whose vertex set is $S$ and whose edge set consists of the edges of $G$ that have both ends in $S$. We denote by $P_k$ a path with $k$ vertices.
We denote by $P_{uv}$ a path with endpoints $u$ and $v$.
The length of a path is defined as the number of its edges.
A block of a graph $G$ is a maximal $2$-connected subgraph of $G$.

The Perron--Frobenius theorem implies that there exists a
non-negative eigenvector of $A(G)$ corresponding to $\lambda(G)$. This vector is called the Perron vector of $G$.
Throughout the paper, we denote by $\bm{x} := ( x_1, \ldots , x_n)^{T}$ the Perron vector of $G$.
Without loss of generality,
we may assume that $\max\{x_v:v\in {V}(G)\}=1$.
This can be achieved by scaling the Perron eigenvector.
The following two facts will be frequently used in our proofs.  The first fact states that for every $i\in V(G)$, we have
$$
\lambda(G)x_i=\sum_{j\in N(i)}x_j.
$$
The second fact concerns with the Rayleigh quotient:
$$\lambda(G)=\max_{\bm{x}\in\mathbb{R}^n}\frac{\bm{x}^{T}A(G)\bm{x}}{\bm{x}^{T}\bm{x}}
=\max_{\bm{x}\in\mathbb{R}^n}\frac{2}{\bm{x}^{T}\bm{x}} \sum_{ij\in E(G)}x_ix_j,$$
where we write $\sum_{ij\in E(G)}$ for the sum over each edge of $E(G)$ once.

 \section{Preliminaries}

 \label{sec-3}

In this section, we present several lemmas that will be utilized in our proof.

\begin{lemma}[See \cite{Bon1971dm,Woo1972}]  \label{lem-Bon-Woo}
Let $G$ be a  graph  on $n$ vertices
with
\[  e(G)\ge  \left\lfloor \frac{n^2}{4} \right\rfloor +1. \]
Then $G$ contains a cycle  $C_{\ell}$
for each $3\le \ell \le \lfloor \frac{n+3}{2} \rfloor$.
\end{lemma}
 Let $g(G)$ and
$c(G)$ be the lengths of a shortest and longest cycle in $G$, respectively.
\begin{lemma}[See \cite{BFG1998}]
\label{lem-weakly-pancyclic}
 If $G$ is a non-bipartite graph with order $n$ and
minimum degree $\delta (G)\ge \frac{n+2}{3}$,
 then $G$ contains  cycles of every length between $g(G)$ and $c(G)$ with $g(G)=3$ or $4$.
\end{lemma}

\begin{lemma}[See \cite{EG1959}] \label{Erdos-Gallai}
If $G$ is an $n$-vertex graph with more than
$\frac{1}{2}t(n-1)$ edges, then $G$ contains a cycle with length at least $t+1$, i.e., $c(G)\ge t+1$.
\end{lemma}

In what follows, we introduce several results about spectral graph theory.

 \begin{lemma}[See \cite{ZLS2021}]
 \label{lem-ZLS}
If $2\ell\in \mathbb{N}$ and $G$ is a graph with $m$ edges and
\[  \lambda (G) > \frac{\ell-1/2 + \sqrt{4m + (\ell-1/2)^2}}{2} , \]
then $G$ contains a cycle of length $t$
for every $t\le 2\ell +2$.
\end{lemma}

In the proof of Theorem \ref{thm-chromatic-spectral-version}, we need to use the following spectral inequality, which can be easily deduced from the result of Sun and Das \cite[Theorem 3.2]{Sun-Das}. We refer to Li and Ning \cite[Lemma 3]{LN2023}.
The importance of Lemma \ref{lem-Sun-Das-deletion} in spectral graph theory has been highlighted in a recent report by Ning \cite{Ning2025}. For example, it could be applied to treat the Nikiforov problem on cycles with consecutive lengths \cite{LN2023,Zhang-wq-2024} and the spectral extremal problem on paths with  given length \cite{Ai2024}.

\begin{lemma}[See \cite{Sun-Das,LN2023}] \label{lem-Sun-Das-deletion}
    Let $G$ be a graph. For every $v\in V(G)$, we have
    \[ \lambda^{2}(G-v)\geq\lambda^{2}(G)-2d(v).\]
\end{lemma}

The following lemma \cite{WXH2005} is also needed in our proof of Theorem \ref{thm-chromatic-spectral-version},
it provides an operation of a connected graph which increases the adjacency spectral radius strictly.

\begin{lemma}[See \cite{WXH2005}] \label{lem-WXH}
Let $G$ be a connected graph
and $(x_1,\ldots ,x_n)^T$ be a Perron vector of $G$,
where $x_i$ corresponds to $v_i$.
Assume that
 $v_i,v_j \in V(G)$ are vertices such that $x_i \ge x_j$, and $S\subseteq N_G(v_j) \setminus N_G(v_i)$ is non-empty.
 Denote $G^*:=G- \{v_jv : v\in S\} +
\{v_iv : v\in S\}$. Then $\lambda (G) < \lambda (G^*)$.
\end{lemma}

  We denote by \( \mathcal{G}_{n,r}^* \) the family of \( n \)-vertex graphs that are obtained from a bipartite graph \( B \) by suspending graphs \( G_1, \ldots, G_s \) for some $s\in \mathbb{N}$ such that \( \sum_{i=1}^s |V(G_i) \setminus V(B)|= r-2 \).
  In other words, there are exactly $r-2$ vertices of $G$ that are not contained in the bipartite subgraph $B$.

\begin{lemma}\label{lem-spectral-order1}
    Let $n\geq 100k$ and $3\le r\leq 10k$ be integers. If $G\in\mathcal{G}^*_{n,r}$, then
    $$
    \lambda(G)\leq \lambda(T_{n-r+2,2}\circ K_{r-1}).
    $$
Moreover, the equality holds if and only if $G=T_{n-r+2,2}\circ K_{r-1}$.
\end{lemma}

\begin{lemma}\label{lem-spectral-order2}
    Let $n\geq 100k$ and $3\le r<s\leq 10k$ be integers. Then
    $$
    \lambda(T_{n-s+2,2}\circ K_{s-1})<\lambda(T_{n-r+2,2}\circ K_{r-1}).
    $$
\end{lemma}

For brevity, the proofs of {Lemmas \ref{lem-spectral-order1} and  \ref{lem-spectral-order2}} will be postponed to the Appendix \ref{App-A}.

\section{Some crucial lemmas}
\label{sec-4}

In this section, we present the key lemmas in our proofs of Theorems \ref{thm-chromatic-edges-version} and \ref{thm-chromatic-spectral-version}.
The first result says that if a bipartite graph has large minimum degree, then for any two distinct vertices $u,v$, we can find a path starting from $v$ to $u$ with any length at most $2k+1$.

 \begin{lemma}\label{lem-edgeversion-greedpath}
    Let $n\ge 100k$ and $G$ be {an} $n$-vertex bipartite graph with minimum degree $\delta(G) {>} \frac{2}{5}n$. Let $v,u\in {V(G)}$ be two distinct vertices of $G$. The following statements hold.
    \begin{itemize}
        \item[\rm (a)] If $v$ and $u$ are in the same part, then for each $h\in\{3,5,\ldots,2k+1\}$, there exists a path from $v$ to $u$ with order $h$.
        \item[\rm (b)] If $v$ and $u$ are in different parts, then for each $s\in \{4,6,\ldots,2k+2\}$, there exists a path from $v$ to $u$ with order $s$.
    \end{itemize}
\end{lemma}

\begin{proof}
Let $V(G)=V_1\sqcup V_2$ be a bipartition of vertices of $G$. For part (a), we may assume that $v,u\in V_1$. Then
    $d_{V_2}(v),d_{V_2}(u)>\frac{2}{5}n$.
 Next, we can greedily find a path of order $h$ starting from $v$ and ending at $u$.
     Indeed, we denote $v_1:=v$ and
     $v_{h}:=u$. Since $d_{V_2}(v_1)>\frac{2}{5}n$,
     we can choose a vertex $v_2\in N_{V_2}(v_1)\setminus \{v_{h+1}\}$.
     Since $v_2\in V_2$, we now
     consider the neighbors of $v_2$ in
     $V_1$. More precisely, a direct computation gives
     \begin{align*}
     |N_{V_1}(v_2)| >\frac{2}{5}n>4k.
     \end{align*}
     Similarly, we can find  $v_3,v_5, \ldots ,v_{h-2}\in V_1$ and
     $v_4,v_6,\ldots ,v_{h-3} \in V_2$
     such that $v_iv_{i+1}\in E({G})$ for each $1\le i\le h-3$. Finally, we show that we can choose a vertex $v_{h-1}\in  V_2$ such that
     $v_{h-2}v_{h-1}, v_{h-1}v_{h}\in E(G)$. This is available for our purpose, since
     $d_{V_2}(v_{h-2}),d_{V_2}(v_{h}) \ge \frac{2}{5}n$. Moreover, we know that $|V_2|< \frac{3}{5}n$.
     Then for $n\ge 100k$,
     \[ |N_{V_2}(v_{h-2}) \cap N_{V_2}(v_{h}) | >\frac{4}{5}n-\frac{3}{5}n>3k\]
Therefore, we can choose a vertex $v_{h-1}\in V_2$ that joins both $v_{h-2}$ and $v_{h}$.
     In conclusion,
     we find a path starting from $v$ to $u$ with order $h$.
Similarly, if $v$ and $u$ are in different partite sets of $G$, we can prove that for each $s\in\{4,6,\ldots,2k+2 \}$, there exists a path from $v$ to $u$ with order $s$.
\end{proof}

\subsection{$k$-dense bipartite subgraph}
 \label{sec-app}

\begin{defn} \label{def-dense}
Let \( G \) be a graph and \( H \) be a subgraph of \( G \). We say that \( H \) is
a {\it $k$-dense bipartite subgraph} of \( G \) if it satisfies the following three conditions.
\begin{itemize}
    \item[(a)]
    $H$ is bipartite and $2$-connected.

    \item[(b)]
    For any two vertices $v, u$ in the same part of \( H \), and for each integer $h\in\{5,\ldots,2k+1\}$, there exists a path from $v$ to $u$ with order $h$ in $H$.

    \item[(c)]
    For any two vertices $v, u$ from different parts of \( H \), and for each integer $s\in\{6,\ldots,2k+2\}$, there exists a path from $v$ to $u$ with order $h$ in $H$.
\end{itemize}
\end{defn}

 The following result shows that every $C_{2k+1}$-free graph on $n$ vertices with at least $\frac{(n-c)^2}{4}$ edges contains a $k$-dense bipartite subgraph with large order and minimum degree. This result plays a key role in our argument. 

\begin{lemma}\label{lem-k-dense-bi-sub}
    Let $c, k\geq 2$, and $n\geq\max\{50c,50k\}$ be integers. If $G$ is {an} $n$-vertex $C_{2k+1}$-free graph with $e(G)\geq \frac{(n-c)^2}{4}$, then
    $G$ contains two subgraphs $F$ and $G'$ with $F\subseteq G'$ such that
    \begin{itemize}
        \item[\rm (a)]
       $|V(F)|\geq n-10c$ and $\delta(F)\geq \frac{2}{5}n$. Moreover, $d_{V(F)}(v)<\frac{2}{5}n$ for any $v\in V(G)\setminus V(F)$.

        \item[\rm (b)]
        $|V(G')|\geq n-2c$ and $\delta(G')\geq 11c$. Moreover, $d_{V(G')}(v)<11c$ for any $v\in V(G)\setminus V(G')$.

        \item[\rm (c)] Both $F$ and $G'$ are $k$-dense bipartite subgraphs.
    \end{itemize}
\end{lemma}
\begin{proof}
    Let $F$ be a subgraph of $G$ defined by a sequence of graphs $H_0,H_1,\ldots,H_s$ such that:
    \begin{algorithm}[H]
\caption{Graph Reduction Algorithm}
\begin{algorithmic}[1]
\STATE Set \( H_0 = G \);
\STATE Set \( i = 0 \);
\WHILE{\( \bm{\delta(H_{i}) \leq 2n/5} \)}
    \STATE Select a vertex $v_i\in V(H_i)$ such that $d_{H_i}(v_i) \le 2n/5$;
    \STATE Set \( H_{i+1} = H_{i} - v_{i} \);
    \STATE Add $1$ to $i$;
\ENDWHILE
\end{algorithmic}
\end{algorithm}
We claim that $s< 10c$. Otherwise, if $s\ge 10c$, then we will  arrive at the graph $H_{10c}$ in the above vertex-deletion process. Note that
\begin{align*}
    e(H_{10c})=e(G)-\sum_{i=0}^{s}d_{H_i}(v_i)
    \geq \frac{(n-c)^2}{4} - 10c \cdot \frac{2n}{5}
    >\frac{(n-10c)^2}{4}.
\end{align*}
Since $n>\max\{50c,50k\}$, we have $|V(H_{10c})|=n-10c>4k$. We know from Lemma \ref{lem-Bon-Woo} that the induced subgraph
$H_{10c}$ contains a copy of $C_{2k+1}$, a contradiction. Therefore, we have $|V(F)|>n-10c$ and $d_F(v)>\frac{2}{5}n$ for each $v\in V(F)$. {By Lemma \ref{Erdos-Gallai},
we see that $F$ contains a cycle with length at least $\frac{2}{5}n +1$.
We claim that $F$ is bipartite.
Otherwise,  $F$ is not bipartite,
then Lemma \ref{lem-weakly-pancyclic}
implies that $F$ contains a cycle $C_t$ for every $4\le t\le \frac{2}{5}n +1$. Since $n\ge 50k$, we find that $F$ contains a copy of $C_{2k+1}$, which leads to a contradiction. So $F$ must be a bipartite graph.}

Let $V(F):= U_1\sqcup U_2$ be a bipartition  of vertices of $F$. Recall that $d_{F}(v)> \frac{2}{5}n$ for each $v\in V(F)$ and the neighbors of $v$ must lie in another vertex part. So we have
 \[ \frac{2}{5}n<|U_1|,|U_2|<\frac{3}{5}n.\]
{We claim that $F$ is $2$-connected.
Next, we show that for any two distinct vertices $u,v\in V(F)$, there are two disjoint paths starting from $u$ to $v$. We divide the argument into two cases.
If $u,v\in U_1$, then
\[ |N_{U_2}(u) \cap N_{U_2}(v)| \ge
|N_{U_2}(u)| + |N_{U_2}(v)| - |U_2| > \frac{n}{5}. \]
There are at least $\frac{n}{5}$ vertices in $U_2$ that are adjacent to both $u$ and $v$. If $u\in U_1$ and $v\in U_2$, then we can choose two distinct vertices $v_1,v_2\in N_{U_2}(u)\setminus \{v\}$. We denote by $U_1'=U_1\setminus \{u\}$.
Note that
\[ |N_{U_1'}(v) \cap N_{U_1'}(v_1)\cap N_{U_1'}(v_2)| > 3 \cdot \frac{2n}{5}  - 2\left(\frac{3n}{5} -1\right) =2.\]
Therefore, we can find two distinct vertices $u_1,u_2\in U_1'$
such that $uv_1u_1v$ and $uv_2u_2v$ are disjoint paths starting from $u$ to $v$.}
By Lemma \ref{lem-edgeversion-greedpath}, we see
that for any two vertices $v,u\in U_i$ with $i\in \{1,2\}$ and any integer $h\in \{3,5,\ldots ,2k+1\}$, we can find a path with order $h$ starting from $u$ to $v$. Moreover, for any $u\in U_1$ and $v\in U_2$, and any $s\in \{4,6,\ldots ,2k\}$, we can find a path of order $s$ starting from $u$ to $v$.
Thus, we conclude that $F$ is a $k$-dense bipartite subgraph of $G$.

Let $G'$ be a subgraph of $G$ defined by a sequence of graphs $G_0, G_1,\ldots,G_t$ such that:

    \begin{algorithm}[H]
\caption{Graph Reduction Algorithm}
\begin{algorithmic}[1]
\STATE Set \( G_0 = G \);
\STATE Set \( i = 0 \);
\WHILE{\( \bm{\delta(G_{i}) < 11c} \)}
    \STATE Select a vertex $u_i\in V(G_i)$ such that $d_{G_i}(u_i) \le 11c$;
    \STATE Set \( G_{i+1} = G_{i} - u_{i} \);
    \STATE Add $1$ to $i$;
\ENDWHILE
\end{algorithmic}
\end{algorithm}

We claim that $t< 2c$. Suppose on the contrary that $t\ge 2c$.  Then we can arrive at the graph $G_{2c}$ in the above reduction algorithm. Thus, we have
\begin{align*}
    e(G_{2c})=e(G)-\sum_{i=0}^{t}d_{G_i}(v_i)
    \geq  \frac{(n-c)^2}{4}
    -2c \cdot 11c
    > \frac{(n-2c)^2}{4} .
\end{align*}
Since $n\geq \max\{50c,50k\}$, we have $|V(G_{2c})|=n-2c>4k$. We know from Lemma \ref{lem-Bon-Woo} that the induced subgraph
$G_{2c}$ contains a copy of $C_{2k+1}$, a contradiction. Therefore, we get $d_{G'}(v)\geq 11c $ for each $v\in V(G')$ and $|G'|\geq n-2c$.
Without loss of generality, we may assume that $F\subseteq G'$. In fact, since $\delta (G_i)< 11c \le 2n/5$, in the process of deleting the vertices,
we can obtain $F$ by first deleting the vertices of $V(G)\setminus V(G')$ and then deleting other vertices with degree at most $2n/5$.

In what follows, we show that $G'$ is bipartite.
Note that $|V(G')\setminus V(F)|< 10c$ and $d_{G'}(v)\geq 11c$ for each $v\in V(G')$. Then $d_{F}(v)\geq c$ for each $v\in {V(G')}$. Recall that $V(F)=U_1\sqcup U_2$.
By Lemma \ref{lem-edgeversion-greedpath}, we obtain that either $d_{U_1}(v) =0$ or $d_{U_2}(v)=0$ for each $v\in V(G')$.
Otherwise, if $v$ has neighbors in both $U_1$ and $U_2$,
then we can find a copy of $C_{2k+1}$, since $F$ is a $k$-dense bipartite graph. This leads to a contradiction.
Thus, one of $d_{U_1}(v)\ge c$ or $d_{U_2}(v)\ge c$ holds for each $v\in V(G')$.  We denote by
\[ V_1 :=\{v\in V(G'):d_{U_2}(v)\geq c\}  \]
and
\[ V_2:=\{v\in V(G'):d_{U_1}(v)\geq c\}. \]
Then $V(G')=V_1\sqcup V_2$ is a partition of vertices of $G'$.

Furthermore, we claim that $e(V_1)=e(V_2)=0$, and so $V(G')=V_1\sqcup V_2$ is a bipartition of $G'$.
Suppose on the contrary that $e(V_1)\geq 1$. Let $vu\in E(V_1)$. By the definition of $V_1$, we have $d_{U_2}(v),d_{U_2}(u)\ge c$.
Let $v_1,v_2\in U_2$ be two distinct vertices satisfy $v_1v,v_2u\in E(G)$.
By Lemma \ref{lem-edgeversion-greedpath}, we can find a path from $v_1$ to $v_2$ with length $2k-2 $ in $F$, then we can find a copy of $C_{2k+1}$ in $G$, a contradiction.
On the other hand, since  each vertex of $V_1$ (resp. $V_2$) has at least $c$ neighbors in $U_2$ (resp. $U_1$), and each vertex of $U_1$ (resp. $U_2$) has at least ${2n \over 5}$ neighbors in $U_2$ (resp. $U_1$),
it is easy to check that $G'$ is 2-connected.
Since  each vertex of $V_1$ (resp. $V_2$) has at least $c$ neighbors in $U_2$ (resp. $U_1$) and \( F \subseteq G' \), by Lemma \ref{lem-edgeversion-greedpath}, we conclude that \( G' \) is a $k$-dense bipartite subgraph of \( G \).
\end{proof}

\subsection{Structural properties outside the bipartition}

\begin{defn} \label{def-bad}
Let $G'$ be a $2$-connected bipartite subgraph of $G$ and
\( P_{uv} \) be a path of \( G \) with endpoints $u,v$.
We call $P_{uv}$ a bad path for $G'$, and $(u,v)$ a bad pair for $G'$, if the endpoints $u,v$ are in \( V(G') \) and all other vertices of \( P_{uv} \) are in \( V(G)\setminus V(G') \), and at least one of the following holds.
\begin{enumerate}
     \item[(a)]  $v, u$ are in the same part of $G'$, and $P_{uv}$ is a path of odd length from $v$ to $u$;

    \item[(b)]  $v, u$ are in different parts of $G'$, and $P_{uv}$ is a path of even length from $v$ to $u$.
\end{enumerate}
\end{defn}

The following structural lemma demonstrates that after obtaining the $k$-dense bipartite subgraph $G'$ via Lemma \ref{lem-k-dense-bi-sub}, the remaining vertices (or subgraphs) outside $G'$ are structurally suspended on $G'$ or a larger bipartite block than $G'$.
Otherwise, if the block containing $G'$ is not bipartite, then there exists a ``bad'' path whose endpoints lie in $G'$, while all intermediate vertices are outside $G'$.

\begin{lemma}\label{lem-bad-for-H}
    Let $G'$ be a $2$-connected bipartite subgraph of $G$. Let $B$ be a $2$-connected subgraph of $G$ that contains $G'$ as a subgraph. If $B$ is not bipartite, then $B$ contains a bad path $P_{uv}$ for $G'$.
\end{lemma}

\begin{proof}
We proceed with the proof by induction on \(|V(B) \setminus V(G')|\). For convenience, we denote $p:=|V(B) \setminus V(G')|$.
When \(p = 1\), that is, \(B\) has exactly one vertex outside of \(G'\), say $v_b$, the conclusion holds immediately, since
$B$ is not bipartite and
$v_b$ must have at least one neighbor in each partite set of  $G'$. So part (b) in Definition \ref{def-bad} is satisfied.
Suppose that the statement holds for $\le p-1$. Now, we consider the case $p\ge 2$.
{Since $B$ is 2-connected, there exists a path in $B$ whose endpoints lie in $V(G')$ and all other vertices belong to $V(B) \setminus V(G') $. Let $ L = w_1 \cdots w_a $ be such a path.  We may further assume that $|L|$ is minimized over all such possible paths. Since $p=|V(B)\setminus V(G')|$, we see that \(3 \leq |L| \leq p + 2\).
Next, we divide the argument into two cases.
}

{\bf {Case 1. \(|L| = p + 2\).}}

Then \(V(B) = V(G') \cup V(L)\). By the minimality of \(|L|\),  we have \(N(w_i) = \{w_{i-1}, w_{i+1}\}\) for every \(2 \leq i \leq a- 1\). Since \(B\) is non-bipartite, \(L\) must be a bad path for \(G'\).

{\bf {Case 2. \(|L| \leq p + 1\).}}

Let \(B'\) be the subgraph of $G$ induced by the vertices of $G' \cup L$. Clearly, \(B'\) is $2$-connected and \(|V(B') \setminus V(G')| \le p -1\). If \(B'\) is not bipartite, then the induction hypothesis ensures that \(B'\) contains a bad path for \(G'\), as needed. Thus, we may assume that \(B'\) is bipartite.

Note that \(|V(B)\setminus V(B')| \le p-1\). By the induction hypothesis, we know that $B$ contains a bad path \(P = z_1z_2\cdots z_c\) for \(B'\), where \(z_1, z_c \in V(B')\) and \(z_2, z_3, \ldots, z_{c-1} \in V(B)\setminus V(B')\).  Since $B'$ is $2$-connected, for any vertex $v\in \{z_1,z_c\}$, there are two internally disjoint paths in $B'$ joining $v$ to vertices of $G'$. Let  $z_1X_1y_1$ and $z_cX_cy_c$ be two disjoint paths joining $z_1$ to $y_1\in V(G')$ and $z_c$ to $y_c\in V(G')$, respectively, where $X_1,X_c\subseteq V(B')\setminus V(G')$. (If $z_1\in V(G')$, then we set $X_1=\emptyset$ and $y_1=z_1$; if $z_c\in V(G')$, then we set $X_c=\emptyset$ and $y_c=z_c$).

Since $B'$ is bipartite and $P$ is bad for $B'$, it is easy to verify that $y_1X_1PX_cy_c$ is a bad path for $G'$.
For example, we consider the case where $y_1,y_c$ are in different color classes of $G'$.
If $z_1,z_c$ are in the same color class of $B'$, then $P$ has odd length because $P$ is bad for $B'$.
Moreover, the sum of lengths of $y_1X_1z_1$ and $y_cX_cz_c$ is odd since $B'$ is bipartite. Thus, the path $y_1X_1PX_cy_c$ has even length.
If $z_1,z_c$ are in different color classes of $B'$, then $P$ has even length, and the sum of lengths of $y_1X_1z_1$ and $y_cX_cz_c$ is even. So $y_1X_1PX_cy_c$ is a path with even length. Similarly, in the case where $y_1,y_c$ are in the same color class of $G'$, we can check that $y_1X_1PX_cy_c$ has odd length. This completes the proof.
\end{proof}

\section{Proof of Theorem \ref{thm-chromatic-edges-version}}

\label{sec-5}

In this section, we present the  proof of Theorem \ref{thm-chromatic-edges-version}.

\begin{proof}[{\bf Proof of Theorem \ref{thm-chromatic-edges-version}}]
    Let $G$ be a $C_{2k+1}$-free graph on $n$ vertices and $G\notin \mathcal{G}_{n,r}$. We need to prove that $e(G)\le \left\lfloor\frac{(n-r+1)^2}{4}\right\rfloor+{r\choose 2}$,  and  equality holds if and only if $G$ is obtained from $T_{n-r+1,2}$ and $K_r$ by sharing exactly one vertex. Without loss of generality, we may assume that $G$ achieves the maximum number of edges. Our goal is to show that $G$ is obtained from $T_{n-r+1,2}$ and $K_r$ by sharing exactly one vertex. By the maximality of $G$, we obtain that
\begin{equation*}
    e(G)\geq \left\lfloor\frac{(n-r+1)^2}{4}\right\rfloor+{r\choose 2}> \frac{(n-2k)^2}{4}.
\end{equation*}
Setting $c=2k$ in Lemma \ref{lem-k-dense-bi-sub}, there exist subgraphs $F\subseteq G'\subseteq  G$ with $|V(F)|\geq n-20k$,
$\delta(F)\geq \frac{2}{5}n$ and $|V(G'|\geq n-4k$,  $\delta(G')\geq 22k$. Moreover, both $F$ and $G'$ are $k$-dense bipartite subgraph of $G$.
Let $V(F)= U_1\sqcup U_2$ be a partition  of vertices of $F$, and $V(G')= V_1\sqcup V_2$ be a partition  of vertices of $G'$.

Let $B$ be the block of $G$ that contains $G'$ as a subgraph. We will prove that $B$ is bipartite. Suppose on the contrary that $B$ is not bipartite.
By Lemma \ref{lem-bad-for-H}, there exists a bad path $P$ for $G'$. We further assume that $|V(P)|$ is minimal among all such paths. We denote by $V(P)\cap V(G')=\{u,v\}$, $P'=P\setminus\{u,v\}$.  Recall that $G'$ is the $k$-dense bipartite subgraph of $G$, and $P$ is a bad path for $G'$.

\begin{claim} \label{clm-4-1-part}
We have $V(G)=V(G')\sqcup V(P')$.
\end{claim}

\begin{proof}[Proof of Claim \ref{clm-4-1-part}]
Otherwise, if \( V(G) \setminus (V(G')\cup V(P')) \neq \emptyset \), then we can choose a vertex \( w \in V(G)\setminus  (V(G') \cup V(P')) \) and fix a vertex \( z \in U_1 \).  We define
\[ G^* := G - \{wu : u \in N(w)\} + \{wu :  u \in N_{U_2}(z)\} . \]
Since $d_{U_2}(z)\geq \frac{2n}{5}\geq 26k \ge d(w)$, we have  $e(G^*)>e(G)$.   First, $G^*$ does not contain $C_{2k+1}$. Suppose for contradiction that $G^*$ contains a copy of $C_{2k+1}$, denoted by $C$. Clearly $w \in V(C)$. Let $w', w''$ be the two vertices adjacent to $w$ in $C$, where $w', w'' \in U_2$. Since $d_{U_1}(w'), d_{U_1}(w'') \geq 2n/5$, we have $|N_{U_1}(w') \cap N_{U_1}(w'')| \geq n/5$. Therefore, we can find a vertex $w''' \in U_1$ outside $V(C)$ that is adjacent to both $w'$ and $w''$. Then $(V(C) \setminus \{w\}) \cup \{w'''\}$ would form a $C_{2k+1}$ in $G$, leading to a contradiction. Second, $G^* \notin \mathcal{G}_{n,r}$. Indeed, for any graph in $\mathcal{G}_{n,r}$, the number of vertices of its non-bipartite blocks cannot exceed $r-1$. However, in $G^*$, the vertex set $V(G') \cup V(P')$ forms a 2-connected non-bipartite subgraph of $G^*$ with $|V(G') \cup V(P')| > n - 4k > r$, which implies $G^* \notin \mathcal{G}_{n,r}$. This contradicts the maximality of \( G \). Therefore, we get \( V(G) =V (G') \sqcup V(P') \).
\end{proof}

We denote $|V(P)|=h$ and divide the remaining proof into five cases.

 {\bf{Case 1.}}  $h\leq 2k-2$.

Since $G'$ is $k$-dense bipartite subgraph,  no matter whether $v$ and $u$ are in the same part, there is a path from $v$ to $u$ with order $2k+3-h$ in $G'$.
Combining with the path $P$, we can see a copy of $C_{2k+1}$ in $G$, which leads to a contradiction.

 {\bf{Case 2.}} $h=2k-1$.
 We may assume that $v\in V_1$ and $u\in V_2$. Then $e(N_{V_2}(v),N_{V_1}(u))=0$. Since $d_{V_1}(u),d_{V_2}(v)\geq 22k$, we have
$$
e(G-P')\leq \left\lfloor\frac{(n-2k+1)^2}{4}\right\rfloor-d_{V_1}(u)\cdot d_{V_2}(v)+d_{V_1}(u)+d_{V_2}(v),
$$
and
 \[ \sum_{w\in P'}d(w) \le 26k(2k-3)  .\]
Then
 \begin{align*}
     e(G)&=e(G-P')+e(G-P',P')+e(P')\\
     &\leq \left\lfloor\frac{(n-2k+1)^2}{4}\right\rfloor-d_{V_1}(u)\cdot d_{V_2}(v)+d_{V_1}(u)+d_{V_2}(v)+26k(2k-3)\\
     &\leq \left\lfloor\frac{(n-2k+1)^2}{4}\right\rfloor-(1-22k)^2+26k(2k-3)+1\\
     &<\left\lfloor\frac{(n-2k+1)^2}{4}\right\rfloor,
 \end{align*}
a contradiction.

{\bf{Case 3.}} $h=2k$.

 We may assume that $P=v_1v_2\cdots v_{2k-1}v_{2k}$, where $v_1=v$, $v_{2k}=u$ and $v,u\in V_1$. By the minimality of $|V(P)|$, we get $N_{G'}(v_i)\subseteq\{v_1,v_{2k}\}$ and $d_{G'}(v_i)\leq 1$ for each $i\in[3,2k-2]$. Moreover, we have $v_iv_j\notin E(G)$ if $|i-j|\geq 3$ and $|i-j|$ is odd for $i,j\in[2k]$. Then $d_{P'}(v_j)\leq k$ for each $j\in [3,2k-2]$, and
 $d_{P'}(v_2),d_{P'}(v_{2k-1})\leq k-1$ and $N_{G'}(v_2),N_{G'}(v_{2k-1})\subseteq V_1$. We denote by
 \[ N_{V_1}(v_2)\cup N_{V_1}(v_{2k-1})\setminus\{v_1,v_{2k}\}
 :=\{w_1,w_2,\ldots,w_s\}.\]
For any $w\in N_{V_1}(v_2)$ and $w'\in N_{V_1}(v_{2k-1})$, we have $N_{V_2}(w)\cap N_{V_2}(w')=\emptyset$ (otherwise, there will be a $C_{2k+1}$, a contradiction). Then $d_{V_2}(v_1)+d_{V_2}(v_{2k})\leq |V_2|$. Without loss of generality, we may assume $d_{V_2}(v_1)\leq d_{V_2}(v_{2k})$. Therefore, we get $d_{V_2}(w_i)+d_{V_2}(v_{1})\leq |V_2|$ for each $1\leq i\le s$. Then
\begin{align*}
    e(G-P')&\leq |V_2|\cdot (|V_1|-s-2)+d_{V_2}(v_1)+d_{V_2}(v_{2k})+\sum_{i=1}^sd_{V_2}(w_i)\\
    &=|V_2|\cdot (|V_1|-s-2)+d_{V_2}(v_1)+d_{V_2}(v_{2k})-sd_{V_2}(v_1)+\sum_{i=1}^s(d_{V_2}(w_i)+d_{V_2}(v_1))\\
    &\leq |V_2|\cdot (|V_1|-1)-22sk\\
    &\leq  \left\lfloor\frac{(n-2k+1)^2}{4}\right\rfloor-22sk.
\end{align*}
Therefore, we have
 \begin{align*}
     e(G)&=e(G-P')+e(G-P',P')+e(P')\\
     &\leq \left\lfloor\frac{(n-2k+1)^2}{4}\right\rfloor-22sk+ 2k-4+2(s+2)+k(k-1)-1\\
     &= \left\lfloor\frac{(n-2k+1)^2}{4}\right\rfloor-22sk+k^2+k+2s-1\\
     &\leq \left\lfloor\frac{(n-2k+1)^2}{4}\right\rfloor+k^2+k-1\\
     &< \left\lfloor\frac{(n-2k+1)^2}{4}\right\rfloor+{2k\choose 2},
 \end{align*}
which leads to a contradiction.

 {\bf{Case 4.}} $h=2k+1$.

We may assume that $P=v_1v_2\cdots v_{2k+1}$, where $v_1=v$, $v_{2k+1}=u$ and $v\in V_1$, $u\in V_2$. By the minimality of $|V(P)|$, we have that $d_{G'}(v_i)\leq 1$ for each odd $i\in[3,2k-1]$, and $d_{G'}(v_j)=0$ for each even $j\in[4,2k-2]$, and $v_iv_j\notin E(G)$ if $|i-j|\geq 3$ and $|i-j|$ is odd for every $i,j\in[2k+1]$.
Recall that $P'=P\setminus \{v,u\}$.
If $k=2$, then $d_{P'}(v_i)\leq k$ for each $i\in [2,2k]$. If $k\geq 3$, then $d_{P'}(v_i)\leq k+1$ for each $i\in[2,2k]$. Therefore, we get
\[ e(P')\leq \frac{2k^2+3k-7}{2}.\]
This bound holds because $k(2k-1)<2k^2+3k-7$ for $k=2$, and $(k+1)(2k-1)\leq 2k^2+3k-7$ for $k\geq 3$, and $k^2+\frac{5k-7}{2}< {2k\choose 2}$ for $k\geq 2$. We claim that $e(N_{V_1}(v_2),N_{V_2}(v_{2k}))=0$. Otherwise, we can find a copy of $C_{2k+1}$ in $G$, which is a contradiction. Then
$$
e(G-P')\leq \left\lfloor\frac{(n-2k+1)^2}{4}\right\rfloor-d_{V_1}(v_2)\cdot d_{V_2}(v_{2k}).
$$
Therefore, we have
 \begin{align*}
     e(G)&=e(G-P')+e(G-P',P')+e(P')\\
     &\leq \left\lfloor\frac{(n-2k+1)^2}{4}\right\rfloor-d_{V_1}(v_2)\cdot d_{V_2}(v_{2k})+d_{V_1}(v_2)+ d_{V_2}(v_{2k})+k-1+e(P')\\
     &\leq  \left\lfloor\frac{(n-2k+1)^2}{4}\right\rfloor-(1-d_{V_1}(v_2))(1-d_{V_2}(v_{2k}))+k+\frac{2k^2+3k-7}{2}\\
     &\leq \left\lfloor\frac{(n-2k+1)^2}{4}\right\rfloor+k^2+\frac{5k-7}{2}\\
     &<\left\lfloor\frac{(n-2k+1)^2}{4}\right\rfloor+{2k\choose 2},
 \end{align*}
which is a contradiction.

{\bf{Case 5.} $2k+2\leq h\leq 4k+2$}.

We denote by $s:=h-2k$. Then $2\leq s\leq 2k+2$. Let  $P=v_1v_2\cdots v_{2k+s-1}v_{2k+s}$, where $v_1=v$ and  $v_{2k+s}=u$. We now consider the case where $s$ is even and $v,u\in V_1$. The case for odd $s$ follows analogously. By the minimality of $|V(P)|$, we have $d_{G'}(v_i)\leq 1$ for each $i\in[3,2k+s-2]$, and $v_iv_j\notin E(G)$ if $|i-j|\geq 3$ and $|i-j|$ is odd for $i,j\in[2k+s]$. Then $d_{P'}(v_j)\leq k+s/2$ for each $j\in [3, 2k+s-2]$,
and $d_{P'}(v_2),d_{P'}(v_{2k+s-1})\leq k+s/2-1$. Consequently, we have
\begin{align*}
    e(G-P')\leq \left\lfloor\frac{(n-2k-s+2)^2}{4}\right\rfloor,
\end{align*}
and
$$
e(P')\leq \frac{(2k+s-2)(k+s/2)}{2}=
\left(k+\frac{s}{2}-1 \right) \left(k+\frac{s}{2} \right).
$$
Therefore, we have
\begin{align*}
    e(G)&=e(G-P')+e(G-P',P')+e(P')\\
    &\leq \left\lfloor\frac{(n-2k-s+2)^2}{4}\right\rfloor+d_{V_1}(v_2)+d_{V_1}(v_{2k+s-1})+2k+s-4+\left(k+\frac{s}{2}-1 \right) \left(k+\frac{s}{2} \right)\\
    &\leq \left\lfloor\frac{(n-2k-s+2)^2}{4}\right\rfloor+44k-4+(k+\frac{s}{2}+1)(k+\frac{s}{2})\\
    &<\left\lfloor\frac{(n-2k+1)^2}{4}\right\rfloor+{2k\choose 2}-\frac{(s-1)(n-2k+1)}{2}-k^2+47k+ {(s-1)(s+k-1)\over 2}\\
    &< \left\lfloor\frac{(n-2k+1)^2}{4}\right\rfloor +{2k\choose 2}-\frac{(s-1)(n-5k)}{2}+45k\\
    &\leq \left\lfloor\frac{(n-2k+1)^2}{4}\right\rfloor+{2k\choose 2}-\frac{n-5k}{2}+45k\\
    &< \left\lfloor\frac{(n-2k+1)^2}{4}\right\rfloor+{2k\choose 2},
\end{align*}
where the last inequality holds since $n\geq 95k$. This leads to a contradiction.

In conclusion, we have proved that $B$ is bipartite. Then $G$ is obtained from a bipartite graph $B$ by suspending some small graphs.
We denote \( t:=|V(G)\setminus V(B)|\). By Lemma \ref{lem-k-dense-bi-sub}, we have $t \leq 4k$ and $e(G)\le \lfloor \frac{(n-t)^2}{4} \rfloor + {t+1 \choose 2}$, which together with \( e(G) \ge \lfloor \frac{(n-r+1)^2}{4} \rfloor +\binom{r}{2}\) yield $t\le r-1$. If $t\le r-2$, then $G\in \mathcal{G}_{n,r}$,  a contradiction. Therefore, we have $t=r-1$ and $e(G) = \lfloor \frac{(n-r+1)^2}{4} \rfloor +\binom{r}{2}$.
So $G$ is obtained from the Tur\'{a}n graph
$T_{n-r+1,2}$ by suspending a clique $K_r$.
\end{proof}

\section{Proof of Theorem \ref{thm-chromatic-spectral-version}}
\label{sec-6}

In this section, we are ready to prove Theorem \ref{thm-chromatic-spectral-version}.

\begin{proof}[{\bf Proof of Theorem \ref{thm-chromatic-spectral-version}}]
Let $3\leq r \leq 2k, k\geq 2$ and $n\geq 712k$.
{Let \( G \) be an \( n \)-vertex graph that contains no copy of \( C_{2k+1} \) and $G\notin \mathcal{G}_{n,r}$.
We need to prove that $\lambda (G)\le \lambda (T_{n-r+1, 2}\circ K_r)$,  and the equality holds if and only if
$G=T_{n-r+1,2}\circ K_r$. Without loss of generality, we may assume that $G$ achieves the maximum spectral radius.
Our goal is to show that $G=T_{n-r+1,2}\circ K_r$.
}
 Since $\lambda (G)$ is maximum, we get
\begin{equation} \label{eq-spectral-ge} \lambda(G)\geq \lambda(T_{n-r+1,2}\circ K_r) >
\lambda (T_{n-r+1,2}) =
\sqrt{\left\lfloor \frac{(n-r+1)^2}{4} \right\rfloor}.
\end{equation}
Since $3\le r\le 2k$, we get $\lambda (G)> \lfloor \frac{n-r+1}{2}\rfloor \ge \frac{n}{2} -k$.
Since $G$ is a $C_{2k+1}$-free graph with $m$ edges,
setting $\ell = k- \frac{1}{2}$ in Lemma \ref{lem-ZLS}, we have
$$\frac{n}{2}-k < \lambda(G)\leq \frac{k- 1+\sqrt{4m+(k- 1)^2}}{2}.$$
It follows that
\begin{equation}
    e(G)>  \frac{n^2}{4}-\frac{1}{2}(3k-1)n + k(2k-1)
    \geq \frac{n^2}{4}-2kn+2k^2>\frac{(n-5k)^2}{4}.
\end{equation}
Setting $c=5k$ in Lemma \ref{lem-k-dense-bi-sub},
there exist subgraphs $F\subseteq G'\subseteq  G$ with $|V(F)|\geq n-50k$, $|V(G')|\geq n-10k$, $\delta(F)\geq \frac{2}{5}n$ and $\delta(G')\geq 55k$. Moreover, both $F$ and $G'$ are $k$-dense bipartite subgraph of $G$.
Let $V(F)= U_1\sqcup U_2$ be a partition of $F$, and $V(G')= V_1\sqcup V_2$ be a partition of $G'$.

\begin{claim}\label{Clm-vector-size}
Let $\bm{x} = ( x_1, \ldots , x_n)^{T}$ be the Perron eigenvector of $G$ corresponding to $\lambda (G)$.  \begin{itemize}
\item[\rm (a)]
    For every $u\in V(G)\setminus V(G')$, we have $x_u< \frac{13}{71}$.

\item[\rm (b)]
 For every $v\in V(F)$, we have $x_v\geq\frac{26}{71}$.
    \end{itemize}
\end{claim}

\begin{proof}[Proof of Claim \ref{Clm-vector-size}]
By Lemma \ref{lem-k-dense-bi-sub}, for every $u\in V(G)\setminus V(G')$, we have
\[ |N_{G}(u)| < 55k +10k =65k.\]
Then
\[ \lambda(G)x_{u} =\sum_{w\in N_G(u)} x_w < 65k. \]
Recall that $\lambda (G)> \frac{n}{2} -k$.
This implies that $x_{u} < \frac{65k}{n/2 -k} \le \frac{13}{71}$ since $n\ge 712 k$.

In what follows, we prove the part (b).
Let $z \in V(G)$ be a vertex such that $x_z$ is the maximum coordinate of the Perron vector $\bm{x}$, i.e.,  $x_z=1$.
 It is easy to verify that \( z \in V(F) \). {Otherwise,
 if $z \in V(G) \setminus V(F)$, then by the construction of $F$, we get $|N_G(z)| < \frac{2n}{5} +50k$.
 It follows that
 \[ \lambda (G) = \lambda(G) x_z = \sum_{w\in N_G(z)} x_w < \frac{2}{5}n + 50k , \]
 which contradicts with \( \lambda(G) > \frac{1}{2}n - k \) since $n\ge 712k$. Recall that $F$ is a bipartite subgraph of $G$ and $V(F)=U_1\sqcup U_2$ is the bipartition of $F$.
Without loss of generality, we may assume that \( z \in U_1 \).

{\bf Step 1.} For every \( v \in U_1 \), we will show that $x_v \ge \frac{65}{142}$.

Since $|N_{U_2}(z)|\ge \frac{2n}{5}$ and $|U_2|< \frac{3n}{5}$, we have
\[ |N_{U_2}(z)| - |N_{U_2}(v) \cap N_{U_2}(z)| =
{-|N_{U_2}(v)|} + |N_{U_2}(v) \cup N_{U_2}(z)| < \frac{n}{5}. \]
Therefore, we get
\begin{align*}
    \lambda(G)x_v - \lambda (G) x_z
    &=\sum_{w\in N_G(v)}x_w - \sum_{w\in N_G(z)}x_w \\
    &\ge -\sum_{w\in N_{U_2}(z)\setminus N_{U_2}(v)} x_w -
    \sum_{w\in V(G)\setminus V(F)} x_w\\
    &\geq - (|N_{U_2}(v)| - |N_{U_2}(v) \cap N_{U_2}(z)|) - 50k \\
    &\geq -\frac{n}{5} - 50k.
\end{align*}
{Since $n\ge 712k$,} it follows that for every $v\in U_1$,
\[ x_v\geq x_z-\frac{n/5 + 50k}{\lambda(G)}\geq
 1- \frac{n+250k}{5(n/2 -k)} \ge  \frac{65}{142}. \]

{\bf Step 2.} For every \( v \in U_2 \), we will prove that $x_v \ge \frac{26}{71}$.

By the definition of $F$, we have
$d_{U_1}(v) \ge \frac{2n}{5}$. By the above argument, we can see that $x_w \ge \frac{65}{142}$ for every $w\in U_1$.
Therefore, we get
$$
\lambda(G)x_v =\sum_{w\in N_G(v)} x_w \geq \sum_{w\in N_{U_1}(v)}x_w\geq \frac{65}{142}\cdot \frac{2n}{5}= \frac{13n}{71}.
$$
Since $G$ is $C_{2k+1}$-free, we have $\lambda (G)\le \lambda (T_{n,2}) \le \frac{n}{2}$.
It follows that $x_v\geq \frac{26}{71}$.}
\end{proof}

 Let $B$ be the block of $G$ that contains $G'$ as a subgraph. Our goal is to prove that $B$ must be bipartite. Suppose on the contrary that $B$ is not bipartite.
By Lemma \ref{lem-bad-for-H}, there exists a bad path $P$ for $G'$. Without loss of generality,
we may assume that $|V(P)|$ is minimal among all such paths. We denote by $V(P)\cap V(G')=\{u,v\}$, $P'=P\setminus\{u,v\}$ and $|V(P)|=h$.

\begin{claim} \label{clm-partition}
 We have $V(G)=V(G')\sqcup V(P')$.
 \end{claim}

 \begin{proof}[Proof of Claim \ref{clm-partition}]
If \( V(G) \setminus (V(G')\cup V(P')) \neq \emptyset \), then we can choose a vertex  \( w \in V(G)\setminus  (V(G') \cup V(P')) \) and fix a vertex \( z \in U_1 \).  By Lemma \ref{lem-k-dense-bi-sub}, we have $|N_{G}(w)| < 55k +10k =65k$.
Define
\[ G^* := G - \{wu : u \in N(w)\} + \{wu : u \in N_{U_2}(z)\} . \]
By Claim \ref{Clm-vector-size} and $n\ge 444k$, we have
\begin{align*}
 \lambda(G^*)-\lambda(G) &\ge \frac{2}{{\bm{x}^T}{\bm{x}}}
    \left(\sum_{ij\in E(G^*)} x_ix_j - \sum_{ij\in E(G)} x_ix_j  \right)\\
    &=\frac{2}{{\bm{x}^T}{\bm{x}}}
    \left(\sum_{u\in N_{U_2}(z)} x_wx_u - \sum_{u\in N(w)} x_wx_u  \right)\\
    &\geq \frac{2x_w}{{\bm{x}^T}{\bm{x}}}
    \left(\frac{2}{5}n \cdot \frac{26}{71}-65k  \right)\\
    &>0.
\end{align*}
First, $G^*$ does not contain $C_{2k+1}$. Suppose in contradiction that $G^*$ contains a copy of $C_{2k+1}$, denoted by $C$. Clearly $w \in V(C)$. Let $w', w''$ be the two vertices adjacent to $w$ in $C$, where $w', w'' \in U_2$. Since $d_{U_1}(w'), d_{U_1}(w'') \geq 2n/5$, we have $|N_{U_1}(w') \cap N_{U_1}(w'')| \geq n/5$. Therefore, we can find a vertex $w''' \in U_1$ outside $V(C)$ that is adjacent to both $w'$ and $w''$. Then $(V(C) \setminus \{w\}) \cup \{w'''\}$ would form a $C_{2k+1}$ in $G$, leading to a contradiction. Second, $G^* \notin \mathcal{G}_{n,r}$. Indeed, for any graph in $\mathcal{G}_{n,r}$, the number of vertices of its non-bipartite blocks cannot exceed $r-1$. However, in $G^*$, the vertex set $V(G') \cup V(P')$ forms a $2$-connected non-bipartite subgraph of $G^*$ with $|V(G') \cup V(P')| > n - 10k > r$, which implies $G^* \notin \mathcal{G}_{n,r}$. This contradicts the maximality of \( G \). Therefore, we conclude that $V(G)=V(G')\sqcup V(P')$.
\end{proof}

{\bf{Case 1.}}  $h\leq 2k-2$.

Since $G'$ is $k$-dense bipartite subgraph,  no matter whether $v$ and $u$ are in the same part, there is a path from $v$ to $u$ with order $2k+3-h$ in $G'$. Therefore, there exists a copy of $C_{2k+1}$ in $G$, a contradiction.

 {\bf{Case 2.}} $h=2k-1$.

 We may assume that $v\in V_1$ and $u\in V_2$. Then $e(N_{V_2}(v),N_{V_1}(u))=0$ {(Otherwise, there would be a $C_{2k+1}$.)}. Note that $d_{V_1}(u),d_{V_2}(v)\geq 55k$. By Lemma \ref{lem-Sun-Das-deletion}, we have
$$
e(G-P')\leq \left\lfloor\frac{(n-2k+1)^2}{4}\right\rfloor-d_{V_1}(u)\cdot d_{V_2}(v)+d_{V_1}(u)+d_{V_2}(v),
$$
and
 \[ \sum_{w\in P'}d(w) \le 65k(2k-3)  .\]
Recall that $G'=G-P'$ is a bipartite graph, then $\lambda(G-P')\leq \sqrt{e(G-P')}$. By Lemma \ref{lem-Sun-Das-deletion}, we have
 \begin{align*}
     \lambda^2(G)&\leq \lambda^2(G-P')+2\sum_{w\in P'}d(w)\\
     &\leq \left\lfloor\frac{(n-2k+1)^2}{4}\right\rfloor-d_{V_1}(u)\cdot d_{V_2}(v)+d_{V_1}(u)+d_{V_2}(v)+130k(2k-3)\\
     &\leq \left\lfloor\frac{(n-2k+1)^2}{4}\right\rfloor-(1-55k)^2+130k(2k-3)+1\\
     &<\left\lfloor\frac{(n-2k+1)^2}{4}\right\rfloor,
 \end{align*}
which contradicts with (\ref{eq-spectral-ge}).

In the sequel, we consider the case \( h \geq 2k \).
In this case, we will construct a graph \( G^* \in \mathcal{G}^*_{n,t} \) where \( t \geq 2k \), and we will demonstrate that \( \lambda(G) < \lambda(G^*) \). By Lemmas \ref{lem-spectral-order1} and \ref{lem-spectral-order2}, we have $\lambda(G) < \lambda(G^*) \leq \lambda(T_{n-r+1,2} \circ K_r) $,
which leads to a contradiction with (\ref{eq-spectral-ge}).
We divide the remaining proof into the following cases.

{\bf{Case 3.}} $h=2k$.

 We may assume that $P=v_1v_2\cdots v_{2k-1}v_{2k}$, where $v_1=v$, $v_{2k}=u$ and $v,u\in V_1$. By the minimality of $|V(P)|$, we have $N_{G'}(v_i)\subseteq \{v_1,v_{2k}\}$ and $d_{G'}(v_i)\leq 1$ for each $i\in[3,2k-2]$.

 For $i,j\in[2k]$, if $|i-j|\geq 3$ and $|i-j|$ is odd, then $v_iv_j\notin E(G)$. Observe that $N_{G'}(v_2),N_{G'}(v_{2k-1})\subseteq V_1$ {(Otherwise, there would be a bad path shorter than $P$)}. We denote by $N:=N_{V_1}(v_2)\cup N_{V_1}(v_{2k-1})\setminus\{v_1,v_{2k}\}=\{w_1,w_2,\ldots,w_s\}$.
For any $w\in N_{V_1}(v_2)$ and $w'\in N_{V_1}(v_{2k-1})$, we have $N_{V_2}(w)\cap N_{V_2}(w')=\emptyset$ {(Otherwise, there would be a $C_{2k+1}$.)}.

 Without loss of generality, we assume that $x_{v_1}\geq x_{v_{2k}}$.   Since $d_{G'}(v_1),d_{G'}(v_{2k})\geq 55k$, $N(v_1)\cap F\neq\emptyset$ and $N(v_2)\cap F\neq \emptyset$. By Claim \ref{Clm-vector-size}, we have
 \[ {\min \left\{\sum_{w\in N_{G'}(v_1)}x_w,  \sum_{w\in N_{G'}(v_{2k})}x_w \right\}>x_{v_{2k-1}}+x_{v_{2}}}. \]
 Without loss of generality, we assume that $\min\Big\{\sum\limits_{w\in N_{G'}(v_1)}x_w,  \sum\limits_{w\in N_{G'}(v_{2k})}x_w \Big\}
 =\sum\limits_{w\in N_{G'}(v_{2k})}x_w$.

To deduce a contradiction, we construct the following graph:
 \begin{align*}
     G^*&:=G-\{v_{2k}w: w\in N_{G'}(v_{2k})\}-\{wv':w\in\{v_2,v_{2k-1}\}, \ v'\in N\} \\
 &\quad +\{v'w: w\in N\cup\{v_1\}, \ v'\in N_{G'}(v_1)\cup N_{G'}(v_{2k})\}.
 \end{align*}
Applying the Rayleigh quotient, we get
\begin{align*}
    &\lambda(G^*)-\lambda(G) \ge \frac{2}{{\bm{x}^T}{\bm{x}}}
    \left(\sum_{ij\in E(G^*)} x_ix_j - \sum_{ij\in E(G)} x_ix_j  \right)
    \\ &\geq\frac{2}{{\bm{x}^T}{\bm{x}}}\left(
    \sum_{w\in N}x_w\sum_{w'\in N_{G'}(v_{2k})} x_{w'}+\sum_{w'\in N_{G'}(v_{2k})}x_{v_1}x_{w'}
    - \sum_{w\in N}(x_{v_2}+x_{v_{2k-1}})x_w-\sum_{w'\in N_{G'}(v_{2k})}x_{v_{2k}}x_{w'} \right)\\
    &=\frac{2}{{\bm{x}^T}{\bm{x}}}\left(
    \sum_{w\in N}x_w\left(\sum_{w'\in N_{G'}(v_{2k})} x_{w'}-(x_{v_2}+x_{v_{2k-1}})\right)+\sum_{w'\in N_{G'}(v_{2k})}(x_{v_1}-x_{v_{2k}})x_{w'} \right)\\
    &>0.
\end{align*}
We conclude that $\lambda (G)< \lambda (G^*)$.
From the definition, we see that $G^*\in \mathcal{G}^*_{n,2k+1}$.
By Lemmas \ref{lem-spectral-order1} and \ref{lem-spectral-order2}, we have $\lambda(G^*)<\lambda(T_{n-2k+1,2}\circ K_{2k})\leq\lambda(T_{n-r+1,2}\circ K_r)$.
Therefore, we get $\lambda (G)< \lambda(T_{n-r+1,2}\circ K_r)$, which is a contradiction.

 {\bf{Case 4.}} $h=2k+1$.

We may assume that $P=v_1v_2\cdots v_{2k+1}$, where $v_1=v$, $v_{2k+1}=u$ and $v\in V_1$, $u\in V_2$. By the minimality of $|V(P)|$, we have $N_{G'}(v_i)\subseteq \{v_1,v_{2k+1}\}$ and $d_{G'}(v_i)\leq 1$ for each $i\in[3,2k-1]$.
Clearly, we have $N_{G'}(v_2)\in V_1$, $N_{G'}({v_{2k}})\subseteq V_2$ and $e(N_{V_1}(v_2), N_{V_2}(v_{2k}))=0$.

 We claim that, for each $v'\in N_{V_2}({v_{2k}})$ and $u'\in N_{V_1}(v_2)$,  we have $V_1\setminus N_{V_1}(v_2)\subseteq N_{V_1}(v')$ and $V_2\setminus N_{V_2}(v_{2k})\subseteq N_{V_2}(u')$. Otherwise, there exists a vertex $w'\in N_{V_1}(v_2)$, and $V_2\setminus N_{V_2}(v_{2k})\not\subseteq N_{V_2}(w')$. Let $w\in V_2\setminus N_{V_2}(v_{2k})$ and $w\not\sim w'$. Connect $ww'$, the resulting graph $G'$ satisfies $\lambda(G') > \lambda(G)$, and $G \notin \mathcal{G}_{n,r}$. If $G'$ contains $C_{2k+1}$, this implies that there exists a path of length $2k+1$ from $w$ to $w'$ in graph $G$. Since $w \notin N_{V_2}(v_{2k})$, it means there must exist a shorter bad path {for $G'$}, which contradicts the minimality assumption of $P$. Therefore, $G'$ does not contain $C_{2k+1}$, but this contradicts the maximality assumption of $\lambda(G)$.

 For each $v'\in N_{V_2}({v_{2k}})$ and $u'\in N_{V_1}(v_2)$, the above argument yields
 \[ \lambda(G)x_{v'}\geq \sum_{w\in V_1 \setminus N_{V_1}(v_2)} x_w \ge  \sum_{w\in V_1}x_w-55k. \]
 Similarly, we have $\lambda(G'){x_{u'}}\geq \sum_{w\in V_2}x_w-55k$. Let $z\in U_1$ and $z'\in U_2$. By Claim \ref{Clm-vector-size}, we have
 $x_z\ge \frac{26}{71} $ and $ x_{z'}\ge \frac{26}{71}$.
Moreover, we observe that $\lambda(G) x_{z'} \le \sum_{w\in V_1} x_w + 10k$. Then
 $$
 \lambda(G)x_{v'}\geq  \lambda(G)x_{z'}-{65k},
 $$
 and similarly, we have
$$
\lambda(G)x_{u'}\geq \lambda(G)x_z-{65k}.
$$
It follows that $x_{v'},x_{u'}> \frac{15}{71}>\frac{13}{71}$. Therefore, for each $v'\in N_{V_2}(v_{2k})\cup N_{V_1}(v_2)$ and $w'\in\{v_2,v_{2k}\}$, we have $x_{v'}>x_{w'}$. We may further assume that $x_{v_1}\geq x_{2k+1}$.  We denote by $P'=P\setminus \{v,u\}$. Let
\begin{align*}
     G^*&:=G-\{v_{2k}w: w\in N_{V_2}(v_{2k})\}-\{v_2w:w\in N_{V_1}(v_2)\setminus\{v_1\}\}-\{v_{2k+1}w:w\in N_{P'}(v_{2k+1})\} \\
 &\quad +\{v_{1}w:w\in N_{P'}(v_{2k+1})\}+\{v'w:v'\in N_{V_1}(v_2), \ w\in N_{V_2}(v_{2k})\}.
 \end{align*}
Similarly, we have
\begin{align*}
    &\lambda(G^*)-\lambda(G) \ge \frac{2}{{\bm{x}^T}{\bm{x}}}
    \left(\sum_{ij\in E(G^*)} x_ix_j - \sum_{ij\in E(G)} x_ix_j  \right)
    \\ &\geq\frac{2}{{\bm{x}^T}{\bm{x}}}\left( \sum_{w\in N_{P'}(v_{2k+1})}x_w(x_{v_1}-x_{v_{2k+1}})+\sum_{w\in N_{V_2}(v_{2k})}x_w(x_{v_1}-x_{v_{2k}})+\sum_{w\in N_{V_1}(v_2)\setminus\{v_1\}}x_w(x_{v_{2k+1}}-x_{v_2})\right)\\
    &>0.
\end{align*}
On the other hand, we have $G^*\in \mathcal{G}^*_{n,2k+1}$. By Lemmas \ref{lem-spectral-order1} and \ref{lem-spectral-order2}, we get $\lambda(G)<\lambda(G^*)<\lambda(T_{n-2k+1,2}\circ K_{2k})\leq \lambda(T_{n-r+1,2}\circ K_r)$, which leads to a contradiction.

{\bf{Case 5.} $2k+2\leq h\leq 10k+2$}.

We denote $s=h-2k$, then $2\leq s\leq 8k+2$. Let  $P=v_1v_2\cdots v_{2k+s-1}v_{2k+s}$, where $v_1=v$, $v_{2k+s}=u$. We now only consider the case $s$ is even and $v,u\in V_1$. The case for odd $s$ follows analogously. By the minimality of $|P|$,  we have $N_{G'}(v_i)\subseteq \{v_1,v_{2k+s}\}$ and $d_{G'}(v_i)\leq 1$ for each $i\in[3,2k+s-2]$. Let $v^*\in P'$ be a vertex such that $x_{v^*}=\max\{x_w: w\in P'\}$. Let
\begin{align*}
     G^*&=G-\{v'w: \ v'\in P', w\in N_{G'}(v') \}+\{v^*w:w\in V_1\}.
 \end{align*}
Similarly, we have
\begin{align*}
    \lambda(G^*)-\lambda(G) &\ge \frac{2}{{\bm{x}^T}{\bm{x}}}
    \left(\sum_{ij\in E(G^*)} x_ix_j - \sum_{ij\in E(G)} x_ix_j  \right)
    \\ &\geq\frac{2}{{\bm{x}^T}{\bm{x}}}\left( \sum_{w\in V_1}x_wx_{v^*}-{\sum_{v'\in P'}\sum_{w\in N_{G'}(v')}x_{v'}x_w}\right)\\
     &\geq\frac{2x_{v^*}}{{\bm{x}^T}{\bm{x}}}\left( \sum_{w\in V_1}x_w-\sum_{v'\in P'}d_{G'}(v')\right)\\
      &\geq\frac{2x_{v^*}}{{\bm{x}^T}{\bm{x}}}\left( \lambda(G)-130k\right)\\
    &>0.
\end{align*}
On the other hand, we have $G^*\in \mathcal{G}^*_{n,h-2}$. By Lemmas \ref{lem-spectral-order1} and \ref{lem-spectral-order2}, we have $\lambda(G)<\lambda(G^*)\le\lambda(T_{n-r+1,2}\circ K_{r})$, which leads to a contradiction.

In summary, we have proved that $B$ is bipartite. Then $G$ is obtained from a bipartite graph $B$ by suspending some small graphs. We denote \( t:=|V(G)\setminus V(B)|\). By Lemma \ref{lem-k-dense-bi-sub},
we have $t \leq 10k$ and $\lambda(G)\le \lambda(T_{n-t}\circ K_{t+1})$, which together with assumptions $G\notin \mathcal{G}_{n,r}$ and \( \lambda(G) \ge \lambda(T_{n-r+1,2}\circ K_r) \) yields $t=r-1$, then $\lambda(G)=\lambda(T_{n-r+1,2}\circ K_{r})$ and $G=T_{n-r+1,2}\circ K_r$. This completes the proof.
\end{proof}

 \section{Concluding remarks}

In this paper, we have investigated a structural stability result for $n$-vertex graphs without an odd cycle $C_{2k+1}$.
Consequently,
we have studied the maximum size and the maximum spectral radius over all $C_{2k+1}$-free graphs with chromatic number $\chi (G)\ge r$, where $r\le 2k$.
We proved that $T_{n-r+1,2}\circ K_r$ is the unique spectral extremal graph. This spectral stability  extends a result of Guo, Lin and Zhao \cite{GLZ2021} as well as a result of Zhang and Zhao \cite{ZZ2023}.
Moreover, our result reveals an interesting phenomenon that the spectral extremal graph for this problem is contained in
the classical edge-extremal graphs.
It would be very nice if it could be extended to the corresponding extremal problem when we forbid a general color-critical graphs.
By a result of Simonovits \cite{Sim1966},
we know that if $F$ is a color-critical graph with $\chi (F)=3$, then for sufficiently large $n$, the bipartite Tur\'{a}n graph $T_{n,2}$ is the unique $F$-free graph with the maximum number of edges. For example, we can consider the specific color-critical graphs $F$, say, the book graphs and theta graphs.
The spectral Tur\'{a}n problems of these two graphs were recently studied by Zhai and Lin \cite{ZL2022jgt}. We refer the readers to \cite{LM2025,FL2025} and references therein.
In general, for a color-critical graph $F$ with $\chi (F)=r+1$,
we can naturally study the spectral Tur\'{a}n type problem for the $F$-free graphs $G$ with chromatic number $\chi (G)\ge k$ for every integer $k\ge r+1$.

\appendix
\section{Appendix}

\label{App-A}

Given a graph G, the vertex partition $\Pi: V(G)=V_1\cup V_2\cup\cdots\cup V_k$ is said to be an equitable partition if, for each $u\in V_i,|V_j\cap N(u)|=b_{ij}$ is a constant depending only on $i,j$ $(1\leq i,j\leq k)$. The matrix $B_\Pi=[b_{ij}]_{i,j=1}^k$ is called the quotient matrix of $G$ with respect to $\Pi$.

\begin{lemma}[See \cite{CRS}]\label{lem-eq-partition}
    Let $\Pi\colon V(G) = V_1\cup V_2\ldots\cup V_k$ be an equitable partition of $G$ with quotient matrix $B_{\Pi}$. Then the largest eigenvalue of $B_{\Pi}$ is the spectral radius of $G$.
\end{lemma}

\subsection{Proof of Lemma \ref{lem-spectral-order1}}

\begin{proof}[{\bf Proof of Lemma \ref{lem-spectral-order1}}]
    Recall that $\mathcal{G}_{n,r}^*$ is the family of graphs that are obtained from a bipartite graph $B$ by suspending graphs $G_1,G_2,\ldots ,G_s$ for some $s\in \mathbb{N}$ such that $\sum_{i=1}^s |V(G_i)\setminus V(B)| = r-2$.
    We may assume that \( G \) is a graph that  achieves the maximum spectral radius over all graphs of $\mathcal{G}_{n,r}^*$.
    Our goal is to prove that $G$ is the desired extremal graph $T_{n-r+2,2} \circ K_{r-1}$.

    {We claim that $s=1$.
    Indeed, we can emerge the suspensions into a large suspension, which increases the spectral radius of $G$.
    Assume on the contrary that $G_1$ and $G_2$ are two graphs that are suspended on a bipartite graph $B$.
    Let $v_1$ and $v_2$ be the intersection vertices of $G_1$ and $G_2$, respectively.
    Let $\bm{x}=(x_1,x_2,\ldots ,x_n)^T$ be the Perron vector corresponding to $\lambda (G)$.
    Without loss of generality, we may assume that $x_{v_1}\le x_{v_2}$. We consider the following graph:
    \[  G':= G- \{v_1u: u\in N_{G_1}(v_1)\} + \{v_2u : u\in N_{G_1}(v_1)\}.  \]
    In other words, we transfer the neighbors of $v_1$ in $G_1$ to the vertex $v_2$, i.e., emerging $G_1$ to $G_2$. Clearly, we have $G'\in \mathcal{G}_{n,r}^*$.
    By Lemma \ref{lem-WXH}, we get $\lambda (G)< \lambda (G')$, which is a contradiction. Therefore, there is exactly one graph $G_1$ that is suspended to $B$.
    By the maximality of the spectral radius of $G$,
    we find that $G_1$ is a clique of order $r-1$ and $B$ is a complete bipartite graph. }

    Next, we show that $B=T_{n-r+2,2}$.
   Let $V(B):= B_1\sqcup B_2$ be a bipartition  of vertices of $B$, and $u^*=V(B)\cap V(G_1)$. We may assume that $u^*\in B_1$.  Let $B_1'=B_1\setminus\{u^*\}$, $G_1'=G_1\setminus\{u^*\}$, $|B_1'|=a$ and $|B_2|=b$. Moreover, we have $a+b=n-r+1$.
   Let $\Pi:V(G)=B_1'\cup B_2\cup V(G_1')\cup \{u^*\}$ be an vertex partition of $G$. Obviously, $\Pi$ is an equitable partition of $G$ with quotient matrix
$$B_{\Pi}=\begin{bmatrix}0&b&0&0\\a&0&0&1\\0&0&r-3&1\\ 0&b&r-2&0\end{bmatrix}.
   $$
By Lemma \ref{lem-eq-partition}, we know that $\lambda(G)$ is  the largest root of
\[ f_{a,b}(x):=x^4-(r-3)x^3-((a+1)b+r-3)x^2+(a+1)b(r-3)x+ab(r-2). \]
Similarly, we obtain that $\lambda(T_{n-r+2,2}\circ K_{r-1})$ is the largest root of $f_{\lfloor\frac{n-r+2}{2}\rfloor-1,\lceil\frac{n-r+2}{2}\rceil}(x)$.
For notational convenience,
we denote $\theta = \lambda (T_{n-r+2,2}\circ K_{r-1})$ and $g(x):=f_{\lfloor\frac{n-r+2}{2}\rfloor-1,\lceil\frac{n-r+2}{2}\rceil}(x)$. Then
\begin{align*}
    f_{a,b}(x)-g(x)&=\left(\left\lfloor\frac{(n-r+2)^2}{4}\right\rfloor-(a+1)b\right)x^2-\left(\left\lfloor\frac{(n-r+2)^2}{4}\right\rfloor-(a+1)b\right) (r-3) x\\
    & \ \ \ \ -\left(\left\lfloor\frac{(n-r+2)^2}{4}\right\rfloor-\left\lceil\frac{n-r+2}{2}\right\rceil -ab\right) (r-2)\\
    &=\left(\left\lfloor\frac{(n-r+2)^2}{4}\right\rfloor-(a+1)b\right)(x-r+2)(x+1)+\left(\left\lceil\frac{n-r+2}{2}\right\rceil-b\right) (r-2).
\end{align*}
Since $a+1+b=n-r+2$, we get $\lfloor \frac{(n-r+2)^2}{4}\rfloor \ge (a+1)b$. Note that $\theta = \lambda (T_{n-r+2,2}\circ K_{r-1}) \ge  \lfloor \frac{n-r+2}{2}\rfloor > r-1$. Then
$(\theta -r+2)(\theta +1)> r-2$.
Therefore, we get
\begin{align*}
    f_{a,b}(\theta) - g(\theta)
    &=\left(\left\lfloor\frac{(n-r+2)^2}{4}\right\rfloor-(a+1)b\right)(\theta -r+2)(\theta +1)+
    \left(\left\lceil\frac{n-r+2}{2}\right\rceil-b\right) (r-2) \\
    &\ge \left(\left\lfloor\frac{(n-r+2)^2}{4}\right\rfloor-(a+1)b +\left\lceil\frac{n-r+2}{2}\right\rceil-b \right) (r-2) \\
    &\ge 0.
\end{align*}
It follows that $f_{a,b}(\theta)\ge 0$ and
$\lambda(G) \leq \lambda(T_{n-r+2,2}\circ K_{r-1})$. Moreover, the equality holds if and only if $a=\lfloor\frac{n-r+2}{2}\rfloor-1$ and $b=\lceil\frac{n-r+2}{2}\rceil$, that is, $G=T_{n-r+2,2}\circ K_{r-1}$. This completes the proof.
\end{proof}

\subsection{Proof of Lemma \ref{lem-spectral-order2}}

\begin{proof}[{\bf Proof of Lemma \ref{lem-spectral-order2}}]
We denote by $G=T_{n-s+2,2}\circ K_{s-1}$, which is
the graph obtained by identifying a vertex of $K_{s-1}$ and a vertex of $T_{n-s-2,2}$ belonging to the partite set with smaller size. Let $B=T_{n-s+2,2}$ and $A=K_{s-1}$.  Let $V(B)= B_1\sqcup B_2$ be a bipartition  of vertices of $B$, where $|B_1|\leq |B_2|$.
    Let $\bm{x} = ( x_1, \ldots , x_n)^{T}$ be the Perron vector of $G$. By scaling, we may assume that $\max\{x_v:v\in V(G)\}=1$.
    Let $v^*\in B_1$ be the vertex of $V(A)\cap V(B)$.
    It is easy to see that $x_{v^*}=\max\{x_v:v\in V(G)\}=1$.
    We now select \( s - r \) distinct vertices from the set \( A \setminus \{v^*\} \), which we denote by \( T' \). Let
    \[ G' := G - \{vu : v \in T', u \in N(v)\} + \{vu : v \in T', u \in B_2\}. \]
    Note that $\lambda(G)>\frac{n-s}{2}\geq 45k$ and
    \[ \sum_{u\in B_2}x_u\geq \lambda(G)x_{v^*}- (s-2)\ge \lambda(G)-10k. \]
     Therefore, we get
\begin{align*}
    \lambda(G')-\lambda(G)&\geq \frac{2}{{\bm{x}^T}{\bm{x}}}\left(\sum_{v\in T'}\sum_{u\in B_2}x_vx_u-\sum_{v\in T'}\sum_{u\in N(v)}x_vx_u\right)\\
    &\geq \frac{2}{{\bm{x}^T}{\bm{x}}}\left((\lambda(G)-10k)\sum_{v\in T'}x_v-10k\sum_{v\in T'}x_v\right)\\
    &=\frac{2}{{\bm{x}^T}{\bm{x}}}\left((\lambda(G)-20k)\sum_{v\in T'}x_v\right)\\
    &>0.
\end{align*}
Moreover, we see that \( G' \in \mathcal{G}^*_{n,r} \). By Lemma \ref{lem-spectral-order1}, we conclude
\[
\lambda(G) < \lambda(G') \leq \lambda(T_{n-r+2,2} \circ K_{r-1}),
\]
which completes the proof.

\end{proof}

\end{document}